\documentclass[12pt]{amsart}
\usepackage{amsfonts,amsmath,amssymb,amsthm}
\allowdisplaybreaks[1]
\usepackage{tikz-cd}
\usepackage{tikz-3dplot}
\usepackage{subcaption}
\usepackage{extarrows}
\usepackage{enumerate}
\usepackage[a4paper, margin=1in]{geometry}
\usepackage{hyperref} 
\usepackage{cleveref} 
\usepackage{bbm}
\usepackage{relsize}

\newtheorem{theorem}{Theorem}[section]
\newtheorem{lemma}[theorem]{Lemma}
\newtheorem{proposition}[theorem]{Proposition}
\newtheorem{corollary}[theorem]{Corollary}

\theoremstyle{definition}
\newtheorem{definition}[theorem]{Definition}

\newtheorem{example}[theorem]{Example}

\newtheorem{conjecture}[theorem]{Conjecture}

\theoremstyle{remark}
\newtheorem{remark}[theorem]{Remark}

\hypersetup{
    colorlinks=true,
    linkcolor=black,
    filecolor=magenta,      
    urlcolor=cyan,
    pdftitle={Your Document Title},
    bookmarks=true
}

\begin{document}
\title{Hilbert-Kunz multiplicity of quadrics decreases}
\author{Cheng Meng}
\address{Cheng Meng\\ Yau Mathematical Sciences Center, Tsinghua University, Beijing 100084, China. \emph{Email:} {\rm cheng319000@mail.tsinghua.edu.cn}}

\date{\today}
\begin{abstract}
In this paper we prove that the Green ring of $\mathbb{Z}/p^e\mathbb{Z}$ is the $e$-fold tensor product of the Green ring of $\mathbb{Z}/p\mathbb{Z}$, and this isomorphism is given by $p$-adic expansions of integers. As an application of this isomorphism, we compute the Hilbert-Kunz function and Hilbert-Kunz multiplicity of Fermat quadrics. Then we use Gelfand transform of the Green ring to give an analytical expression of this Hilbert-Kunz multiplicity, and prove that it decreases with its characteristic, thus giving a positive answer to a conjecture of Yoshida.

\end{abstract}
\maketitle
\tableofcontents
\section{Introduction}
\subsection{The Hilbert-Kunz multiplicity}
Let $(R,\mathfrak{m})$ be a Noetherian local ring of characteristic $p$. The Hilbert-Kunz multiplicity $e_{HK}(R)$ is an important invariant for $R$.
\begin{definition}[\cite{Mon83}]
$$e_{HK}(R)=\lim_{e \to \infty}\frac{\ell(R/\mathfrak{m}^{[p^e]})}{p^{e\dim R}}.$$    
\end{definition}
Here $\mathfrak{m}^{[p^e]}=(f^{p^e}|f \in \mathfrak{m})$ is the $e$-th Frobenius power of $\mathfrak{m}$. By the result of \cite{Mon83}, $e_{HK}(R)$ always exists as a positive real number. In general, it is a real number in $[1,\infty)$. By \cite[Theorem 1.5]{WYeHK=1}, for a formally unmixed local ring $R$, $e_{HK}(R)=1$ if and only if $R$ is regular local. So roughly speaking, $e_{HK}(R)$ measures the singularity of the ring $R$.

Concrete computation of the Hilbert-Kunz multiplicity is very hard. In general, the Hilbert-Kunz multiplicity may not be rational by a well-known unpublished result of Brenner \cite{Bre13}. Up to now, there is no reliable algorithm to compute $e_{HK}(R)$ for a ring $R$ with a general explicit representation, and the values of these invariants are only known for very special classes of rings.  For example, the Hilbert-Kunz multiplicity of $\mathbb{F}_2[[x,y,z,u,v]]/(x^3+y^3+xyz+uv)$ is only conjecturally known in \cite{Monsky08conj}.

\subsection{Conjectures by Watanabe and Yoshida}
A natural question is how close the Hilbert–Kunz multiplicity of a singularity can be to 1. It turns out that such Hilbert-Kunz multiplicity cannot be arbitrarily close when $d=\dim R$ is fixed; for example, when $e_{HK}(R)<1+\frac{1}{d(d!(d-1)+1)^d}$, $R$ is regular by \cite[Theorem 4.12]{AE08}, so $e_{HK}(R)=1$. A better result in \cite[Chapter 8]{Huneke} shows that if $e_{HK}(R)<1+\frac{1}{d^dd!}$, then $R$ is regular. So it arouses one's interest to find an exact description of the minimum or infimum for $e_{HK}(R)$ of singular rings $R$ depending on both $p,d$ or solely on $d$.
In 2005, Watanabe and Yoshida made the following conjecture which predicts the ring achieving this minimal Hilbert-Kunz multiplicity, and the behavior of this multiplicity with respect to the characteristic:
\begin{conjecture}[\cite{WYconj05}, Conjecture 4.2]
Let $p$ be an odd prime, $Q_{p,d}=\mathbb{F}_p[[x_1,\ldots,x_d]]/(x_1^2+\ldots+x_d^2)$\footnote{note that here the index $d$ is shifted by 1 compared with other references like \cite{WYconj05} and \cite{PSSY}. The term $A_{p,d}$ in these references refers to $Q_{p,d+1}$.}, which is a singular ring of characteristic $p$ and dimension $d-1$. Then:
\begin{enumerate}
\item for any formally unmixed non-regular local ring $R$ of characteristic $p$ and dimension $d$, $e_{HK}(R)\geq e_{HK}(Q_{p,d+1})$.
\item $e_{HK}(Q_{p,d})\geq \lim_{p \to \infty}e_{HK}(Q_{p,d})$.
\item for any formally unmixed non-regular local ring $R$ of characteristic $p$ and dimension $d$, $e_{HK}(R)=e_{HK}(Q_{p,d+1})$ if and only if $R$ and $Q_{p,d+1}$ are isomorphic up to completion and change of base field.
\end{enumerate}    
\end{conjecture}
Later a stronger form of (2) has been conjectured in an unpublished note of Yoshida \cite{Yos19}:
\begin{enumerate}
\item [(4)] $p \mapsto e_{HK}(Q_{p,d})$ is decreasing in $p$.
\end{enumerate}
If the above conjectures are true, then $e_{HK}(Q_{p,d+1})$ is a lower bound in fixed characteristic and dimension, and $\lim_{p \to \infty}e_{HK}(Q_{p,d+1})$ would be a common lower bound for $e_{HK}(R)$ for all characteristics in fixed dimension, which is a great improvement from the known bound $1+\frac{1}{d^dd!}$.

We list the most recent progress on these conjectures.
\begin{enumerate}
\item when $R$ is a complete intersection ring \cite{ES05} \cite{CR25}, when $\dim(R)=3$\cite{WY01}, $4$\cite[Theorem 4.3]{WYconj05}, $5,6$\cite[Theorem 5.2]{AE13}, $7$ \cite[Theorem 1.2]{AC24}.
\item when $p>d-3$ \cite[Theorem A]{TriWYinequ}, and for all $p$, \cite{Meng89}. A refinement of (2) for $p=2$ is proved in \cite{CR25}. An alternative proof is given by \cite{KAH} using results from \cite{PSSY}.
\item when $R$ is a complete intersection \cite{CR25}, when $2 \leq \dim R \leq 4$ \cite{WYconj05}.
\item When $p\gg 0$ \cite[Theorem C]{TriWYinequ}, when $d \leq 31$ \cite{PSSY}.
\end{enumerate}

\subsection{Results}
Here is the main theorem of this paper.
\begin{theorem}[See \Cref{5.1}]
Conjecture (4) holds. That is, $p \mapsto e_{HK}(Q_{p,d})$ is decreasing.   
\end{theorem}
We will focus on the case $p>2$. For the analogous statement for $p=2$ we need to replace $Q_{2,d}$ with another ring $Q'_{2,d}$. In this case, $e_{HK}(Q'_{2,d})\geq e_{HK}(Q_{3,d})$ by \cite{CR25}.

A sketch of proof is as follows. Along with the proof, we will also present some other results of the paper. We denote $e_{HK}(Q_{p,d})=b_{p,d}$ in the following theorem.

First, it is already clear from Han and Monsky's method in \cite{HM93} that such Hilbert-Kunz multiplicity can be computed using the representation ring $\Gamma$ of $\mathbbm{k}$-objects, or on its subrings $\Gamma_e$. However, the computation of the product depends heavily on combinatorics. We show that there is an isomorphism between $\Gamma_e$ and $e$-th tensor power of $\Gamma_1$:
\begin{theorem}[See \Cref{2.8}]
We have an isomorphism of rings
$$\varphi_e:\Gamma_e \to \Gamma_1\otimes_\mathbb{Z}\Gamma_1\otimes_\mathbb{Z}\ldots \otimes_\mathbb{Z}\Gamma_1=\Gamma_1^{\otimes e}$$
such that
$$\varphi_e(\lambda_a)=\lambda_{b_{e-1}}\otimes \lambda_{b_{e-2}}\otimes \ldots \otimes \lambda_{b_0}$$
whenever
$$\lambda_a=\theta^{e-1}(\lambda_{b_{e-1}})\theta^{e-2}(\lambda_{b_{e-2}})\ldots \theta(\lambda_{b_1})\lambda_{b_0}.$$
Here $\lambda_i=(-1)^i(\delta_{i+1}-\delta_i)$, $\delta_i$ is the unique $i$-dimensional indecomposable $\mathbbm{k}$-object, and $\theta$ is the map
$$\lambda_i \mapsto \begin{cases}
\lambda_{pi} & i \textup{ even}\\
\lambda_{pi+p-1} & i \textup{ odd.}
\end{cases}$$
\end{theorem}
An observation of the author in a previous work says:
\begin{theorem}[\cite{MengBenson}, Theorem 2.12]
$\Gamma_e$ is the Green ring of $\mathbb{Z}/p^e\mathbb{Z}$.    
\end{theorem}
Therefore, the above theorem also gives the structure of the Green ring. The importance of this isomorphism theorem is that it is \textit{explicit in $p$-adic expansions}. This allows us to do concrete computation of the Hilbert-Kunz function and Hilbert-Kunz multiplicity in the representation ring. 

The next step is direct computation, where we recover the following result in \cite{PSSY}. We set $r=(p-1)/2$.
\begin{theorem}[\cite{PSSY}, Corollary 3.3]
$b_{p,d}=1+\frac{\mu((\delta_r+\delta_{r+1})^d)-p^{d-1}}{p^{d-1}-(-1)^{rd}\mu((\delta_{r+1}-\delta_r)^d)}$    
\end{theorem}
Here $\mu$ is the map sending a $\mathbbm{k}$-object to its number of generators. Moreover we also compute the Hilbert-Kunz function:
\begin{theorem}
The Hilbert-Kunz function of $Q_{p,d}$ is $b_{p,d}p^{e(d-1)}+C(e)\mu((-1)^r\lambda_r^d)^e$. Here $C(e)$ is a constant for $d$ even and is periodic of period 2 for $d$ odd.    
\end{theorem}
After we express $e_{HK}(Q_{p,d})$ in terms of products in $\Gamma_1$, we need to compute this product. We apply the previous observation that $\Gamma_1$ is a Green ring together with the result of \cite{Green} which says $\Gamma_1$ is semisimple and gives all homomorphisms from the Green ring to $\mathbb{C}$. This allows us to describe the structure of $\Gamma_1$ using Gelfand transform, which makes the product structure much more explicit. We get the following analytical expression of $e_{HK}(Q_{p,d})$:
\begin{theorem}
We set $r=(p-1)/2$ and
$$S_{p,d}=\sum_{1 \leq i \leq p-1, i \textup{ odd}}\sin(\frac{i\pi}{2p})^{-d},T_{p,d}=\sum_{1 \leq i \leq p-1, i \textup{ odd}}(-1)^{(i-1)/2}\sin(\frac{i\pi}{2p})^{-d}.$$
Then
$$b_{p,d}-1=\begin{cases}
\frac{2S_{p,d}-2S_{p,d-2}}{p^d-2S_{p,d-2}-1} & d \textup{ even}\\
\frac{2T_{p,d}-2T_{p,d-2}}{p^d-2T_{p,d-2}-(-1)^{r}} & d \textup{ odd}.
\end{cases}$$    
\end{theorem}

Now the final result can be proved using analysis. We explicitly proved the case $p \leq 5$ by direct computation and $p \geq 5,d \geq 16$ by an estimate of $S_{p,d}$ or $T_{p,d}$. Since the case $d \leq 15$ is already proved in \cite{PSSY}, the conjecture is confirmed.

Apart from the monotonicity, we also explored the expression of $e_{HK}(Q_{p,d})$ in terms of $p,d$. Note that the result in \cite{PSSY} says $e_{HK}(Q_{p,d})$ is a rational function of $p$ when $d$ is fixed. In this paper, we also find another expression of $S_{p,d}$ and $T_{p,d}$ using residues of meromorphic function which shows that $S_{p,d}$ and $T_{p,d}+\frac{1}{2}(-1)^r$ are polynomials of $p$ for $d \geq 1$. Therefore, the previous theorem leads to a closed formula for this rational function for $d \geq 3$: 
\begin{theorem}[See \Cref{4.7}]
Let $d$ be a fixed integer, and $p \geq 3$ be an odd prime number. We have:
\begin{enumerate}
\item When $d \geq 4$ is even, 
$$b_{p,d}-1=\frac{\sum_{i=1}^{d/2}\frac{(-1)^{i-1} 2^{2i} (2^{2i} - 1) B_{2i}\rho_{d,d-2i}p^{2i}}{(2i)!}-\sum_{i=1}^{(d-2)/2}\frac{(-1)^{i-1} 2^{2i} (2^{2i} - 1) B_{2i}\rho_{d-2,d-2i-2}p^{2i}}{(2i)!}}{p^d-\sum_{i=1}^{(d-2)/2}\frac{(-1)^{i-1} 2^{2i} (2^{2i} - 1) B_{2i}\rho_{d-2,d-2i-2}p^{2i}}{(2i)!}}$$
\item When $d \geq 3$ is odd,
$$b_{p,d}-1=\frac{\sum_{i=0}^{(d-1)/2}\frac{(-1)^i E_{2i}\rho_{d,d-2i-1}p^{2i+1}}{(2i)!}-\sum_{i=0}^{(d-3)/2}\frac{(-1)^i E_{2i}\rho_{d-2,d-2i-3}p^{2i+1}}{(2i)!}}{p^d-\sum_{i=0}^{(d-3)/2}\frac{(-1)^i E_{2i}\rho_{d-2,d-2i-3}p^{2i+1}}{(2i)!}}$$
\end{enumerate}
\end{theorem}
Here $\rho_{d,i}$ is the coefficient of $z^i$ in the Taylor expansion of $z^d/\sin(z)^d$, $B_i,E_i$ are the Bernoulli number and the Euler number. 

A byproduct in the proof of this closed formula is a concrete expression of the Ehrhart polynomials discussed in \cite{PSSY}:

\begin{theorem}[See \Cref{4.8}]
For $d \geq 1$, let $F_d(n)$ and $E_d(n)$ be the Ehrhart polynomials of the $d$-dimensional Fibonacci and extended Fibonacci polytopes. Then we have:
\begin{align*}
F_{d}(\frac{p-3}{2})=\begin{cases}
\frac{1}{2^{d}p}(\sum_{i=1}^{(d+1)/2}\frac{(-1)^{i-1} 2^{2i} (2^{2i} - 1) B_{2i}\rho_{d+1,d-2i+1}p^{2i}}{(2i)!}\\-\sum_{i=1}^{(d-1)/2}\frac{(-1)^{i-1} 2^{2i} (2^{2i} - 1) B_{2i}\rho_{d-1,d-2i-1}p^{2i}}{(2i)!}) & d \textup{ odd}\\
\\
\frac{1}{2^{d}p}(\sum_{i=0}^{d/2}\frac{(-1)^i E_{2i}\rho_{d+1,d-2i}p^{2i+1}}{(2i)!}\\-\sum_{i=0}^{(d-2)/2}\frac{(-1)^i E_{2i}\rho_{d-1,d-2i-2}p^{2i+1}}{(2i)!}) & d \textup{ even}.
\end{cases}
\end{align*}
for $d \geq 2$ and
$$E_{d}(\frac{p-1}{2})=\begin{cases}
\frac{1}{p}(\sum_{i=1}^{(d+1)/2}\frac{(-1)^{i-1} 2^{2i} (2^{2i} - 1) B_{2i}\rho_{d+1,d-2i+1}p^{2i}}{(2i)!}) & d \textup{ odd}\\
\frac{1}{p}(\sum_{i=0}^{d/2}\frac{(-1)^i E_{2i}\rho_{d+1,d-2i}p^{2i+1}}{(2i)!}) & d \textup{ even}.
\end{cases}
$$
for $d \geq 1$.    
\end{theorem}

\subsection{Outline of the paper} This paper consists of 5 sections. Section 2 is devoted to the isomorphism theorem of Green ring $\Gamma_e$, along with a description of the behavior of the dimension function and the number-of-generator function on representations under the isomorphism. Section 3 is devoted to the computation of the Hilbert-Kunz function and Hilbert-Kunz multiplicity of $Q_{p,d}$. In Section 4, we apply Gelfand transform to give the analytical expression of $e_{HK}(Q_{p,d})$ in terms of trigonometric sums. Finally in Section 5, we give estimates of the sums, and prove Yoshida's conjecture.

\section{The Green ring of a cyclic $p$-group}
In this paper, $p$ will be a prime number, $\mathbbm{k}$ will be a field of characteristic $p$. When $p$ is odd, we set $r=(p-1)/2$.

\subsection{$\mathbbm{k}$-object and modular representations}
We first recall the definition of $\mathbbm{k}$-object introduced in \cite{HM93}.

A \textbf{$\mathbbm{k}$-object} is a finitely generated \(\mathbbm{k}[T]\)-module on which \(T\) acts nilpotently. \(\Gamma\) is the group given by the free abelian group over isomorphic classes $[M]$ of $\mathbbm{k}$-objects $M$ modulo the relation $[M\oplus N]-[M]-[N]$. We introduce a product on \(\Gamma\) as follows: if two elements of \(\Gamma\) are represented by \(\mathbbm{k}\)-objects \(M\) and \(N\), then their product is the image in \(\Gamma\) of the \(\mathbbm{k}\)-object \(M \otimes_{\mathbbm{k}} N\), where \(T\) acts distributively; namely \(T(m \otimes n) = (Tm) \otimes n + m \otimes (Tn)\). Now \(\Gamma\) endowed with this product is a commutative ring, called the \textbf{representation ring}. The zero and unity of \(\Gamma\) are respectively the images of the zero module and \(\mathbbm{k}[T]/(T)\) in \(\Gamma\).

For any nonnegative integer \(n\), \(\delta_n\) is the image of \(M_n = \mathbbm{k}[T]/(T^n)\) in \(\Gamma\) (so in particular \(\delta_0 = 0\) and \(\delta_1 = 1\)). The theory of modules over principal ideal domains shows that \((\Gamma, +)\) is a free abelian group with basis \(\{\delta_1, \delta_2, \dots\}\). In what follows we shall also use a second basis \(\{\lambda_0, \lambda_1, \dots\}\), where \(\lambda_n = (-1)^n(\delta_{n+1} - \delta_n)\) and $\delta_i=\lambda_0-\lambda_1+\ldots+(-1)^{i-1}\lambda_{i-1}$. The transition between $\lambda$-basis and $\delta$-basis will become essential in later computations.

\begin{definition}
For $e \in \mathbb{N}$, define
$$\Gamma_e=\sum_{1 \leq i \leq p^e}\mathbb{Z}\delta_i \subset \Gamma,\Gamma_e(\mathbb{C})=\Gamma_e\otimes_\mathbb{Z}\mathbb{C} \subset \Gamma\otimes_{\mathbb{Z}}\mathbb{C}.$$
\end{definition}
We see $\Gamma_e,e \in \mathbb{N}$ is actually a filtration of subrings of $\Gamma$. Each $\Gamma_e$ has two bases $\delta_i,1 \leq i \leq p^e$ and $\lambda_i,0 \leq i \leq p^e-1$. These two bases are also $\mathbb{C}$-bases of $\Gamma_e(\mathbb{C})$. We will mainly work with the group $\Gamma_e(\mathbb{C})$ with $\mathbb{C}$-coefficients in this paper.

Each $\delta_i \in \Gamma_e$ corresponds to a $\mathbbm{k}[T]/(T^{p^e})$-module, which becomes a $\mathbb{Z}/p^e\mathbb{Z}$-representation under ring isomorphism $\mathbbm{k}[\mathbb{Z}/p^e\mathbb{Z}]\cong\mathbbm{k}[T]/(T^{p^e})$. A generator of $\mathbb{Z}/p^e\mathbb{Z}$ acts on the $\mathbbm{k}$-object via $1+T$. In this sense, $\delta_i$ is the unique $i$-dimensional indecomposable representation for $1 \leq i \leq p$. Now we have two products on $\Gamma_e$, namely the product of $\mathbbm{k}$-objects and tensor product of representations. We have the following result. 
\begin{theorem}[\cite{MengBenson}, Theorem 2.12]\label{2.2}
The product as $\mathbbm{k}$-object is isomorphic to the tensor product of representations. Therefore, $\Gamma_e$ is the Green ring of $\mathbb{Z}/p^e\mathbb{Z}$.    
\end{theorem}
We sketch a proof in \cite{Meng89} and \cite{MengBenson} here for reader's convenience.
\begin{proof}
Suppose $M_1$, $M_2$ are two $\mathbbm{k}$-objects in $\Gamma_e$. View them as $\mathbb{Z}/p^e\mathbb{Z}$-representations. Let $M_1$ be a $k[T_1]/(T_1^{p^e})$-module, $M_2$ be a $k[T_2]/(T_2^{p^e})$-module, then $M_1\otimes_k M_2$ is a $k[T_1,T_2]/(T_1^{p^e},T_2^{p^e})$-module. Let $\varphi,\psi: k[T]/(T^{p^e}) \to k[T_1,T_2]/(T_1^{p^e},T_2^{p^e})$ be ring homomorphisms such that $\varphi(T)=T_1+T_2$ and $\psi(T)=T_1+T_2+T_1T_2$, then one can verify that the tensor product as $k$-object is $\varphi_*(M_1\otimes_k M_2)$ and as representation is $\psi_*(M_1\otimes_k M_2)$. Since $k[T_1],k[T_2]$ are PID's, it suffices to verify when $M_1,M_2$ are cyclic. In this case $M_1=k[T_1]/(T_1^a), M_2=k[T_2]/(T_2^b), M_1\otimes_kM_2=k[T_1,T_2]/(T_1^a,T_2^b)$. However, there is an automorphism $\rho$ of the ring $k[T_1,T_2]/(T_1^{p^e},T_2^{p^e})$ sending $T_1$ to $T_1$ and $T_2$ to $(1+T_1)T_2$. In this case $\rho\varphi=\psi$. Moreover, $(T_1^a,T_2^b)$ is stable under the action of $\rho$. So $\rho_*$ induces an isomorphism between $\varphi_*(M_1\otimes_k M_2)$ and $\psi_*(M_1\otimes_k M_2)$.
\end{proof}

Here the Green ring of a finite group $G$ is just the free abelian group over isomorphic classes of indecomposable modules. Its addition corresponds to direct sums of modules and multiplication corresponds to the tensor product of representations. 

\Cref{2.2} implies that the following computational results on the product of $\Gamma_e$ are equivalent:
\begin{enumerate}
\item From representation side, \cite{Green}\cite{Srinivasan64}\cite{Renaud79}\cite{BensonBanach}.
\item From commutative algebra side, \cite{HM93}\cite{TeixeiraII}.
\end{enumerate}

We will use the following two results on the product structure of $\Gamma_1$.
\begin{proposition}[\cite{HM93}, Theorem 2.5]\label[proposition]{2.3}
For $0\leq i \leq j<p$,
$$\lambda_i\lambda_j=\sum_{k=j-i}^{\min\{i+j,2p-2-i-j\}}\lambda_k.$$    
\end{proposition}
\begin{proposition}[\cite{HM93}, Theorem 3.4]\label[proposition]{2.4}
For any $0 \leq i<p$ and $j \in \mathbb{N}$, we have:
\begin{enumerate}
\item $\lambda_i\lambda_{pj}=\lambda_{pj+i}$;
\item if $j\neq 0$, $\lambda_i\lambda_{pj-1}=\lambda_{pj-1-i}$.
\end{enumerate}
\end{proposition}

\subsection{Tensor product structure of $\Gamma_e(\mathbb{C})$}
For a nonnegative integer $n$, there is a unique finite $p$-adic expansion
$$n=a_0+a_1p+a_2p^2+\ldots+a_ip^i, 0 \leq a_0,\ldots,a_i<p.$$
In this case we write $n=(a_i,\ldots ,a_1,a_0)_p$. If each $a_i$ is represented with a single letter with indices, we omit the comma here and write $(a_i\ldots a_1a_0)_p$. We denote $i^*=p-1-i$ for $i \in \{0,1,\ldots,p-1\}$ and  
$$i^{\sigma(n)}=\begin{cases}
i & n \textup{ even}\\
i^* & n \textup{ odd.}
\end{cases}$$
We write $\sigma(a_i\ldots a_1a_0)=\sigma((a_i\ldots a_1a_0)_p)$.
\begin{definition}
Let $\theta:\Gamma \to \Gamma$ be the map
$$\lambda_i \mapsto \begin{cases}
\lambda_{pi} & i \textup{ even}\\
\lambda_{pi+p-1} & i \textup{ odd.}
\end{cases}$$
\end{definition}
\begin{theorem}[\cite{TeixeiraII}, Theorem 2.13]
$\theta$ is a ring homomorphism.    
\end{theorem}

\begin{proposition}
Let $a=(a_{e-1}\ldots a_1a_0)_p$ be an integer between $0$ and $p^e-1$. Then we have
$$\lambda_a=\theta^{e-1}(\lambda_{a_{e-1}})\theta^{e-2}(\lambda_{a_{e-2}^{\sigma(a_{e-1})}})\ldots \theta(\lambda_{a_1^{\sigma(a_{e-1}\ldots a_2)}})\lambda_{a_0^{\sigma(a_{e-1}\ldots a_1)}}.$$
Moreover it is the unique equation of the form
$$\lambda_a=\theta^{e-1}(\lambda_{b_{e-1}})\theta^{e-2}(\lambda_{b_{e-2}})\ldots \theta(\lambda_{b_1})\lambda_{b_0}$$
where $b_{e-1},\ldots,b_0 \in \{0,1,\ldots,p-1\}$.
\end{proposition}
\begin{proof}
By \Cref{2.4} we see
$$\theta(\lambda_i)\lambda_{j^{\sigma(i)}}=\lambda_{pi+j},\theta(\lambda_i)\lambda_j=\lambda_{pi+j^{\sigma(i)}}.$$
So
\begin{align*}
\theta^2(\lambda_i)\theta(\lambda_{j^{\sigma(i)}})\lambda_{k^{\sigma(ij)}}=\theta(\lambda_{pi+j})\lambda_{k^{\sigma(ij)}}
=\lambda_{p^2i+pj+k},\ldots.
\end{align*}
Applying this equation inductively, we see the first equation holds. On the other hand we have
$$\theta^2(\lambda_i)\theta(\lambda_j)\lambda_k=\theta(\lambda_{pi+j^{\sigma(i)}})\lambda_k
=\lambda_{p^2i+pj^{\sigma(i)}+k^{\sigma(ij^{\sigma(i)})}},\ldots.$$
Inductively we see
$$\theta^{e-1}(\lambda_{b_{e-1}})\theta^{e-2}(\lambda_{b_{e-2}})\ldots \theta(\lambda_{b_1})\lambda_{b_0}=\lambda_a$$
if and only if
$$b_{e-1}=a_{e-1},b_{e-2}^{\sigma(b_{e-1})}=a_{e-2},b_{e-3}^{\sigma(b_{e-1}b_{e-2}^{\sigma(b_{e-1})})}=a_{e-3},\ldots$$
since $0 \leq a<p^e$ and $0 \leq b_i<p$, we can prove inductively that there is only one such solution of $b_i$'s to the above equation.
\end{proof}
\begin{theorem}\label{2.8}
We have an isomorphism of rings
$$\varphi_e:\Gamma_e \to \Gamma_1\otimes_\mathbb{Z}\Gamma_1\otimes_\mathbb{Z}\ldots \otimes_\mathbb{Z}\Gamma_1=\Gamma_1^{\otimes e}$$
such that $\varphi_e$ is $\mathbb{Z}$-linear and
$$\varphi_e(\lambda_a)=\lambda_{b_{e-1}}\otimes \lambda_{b_{e-2}}\otimes \ldots \otimes \lambda_{b_0}$$
whenever
$$\lambda_a=\theta^{e-1}(\lambda_{b_{e-1}})\theta^{e-2}(\lambda_{b_{e-2}})\ldots \theta(\lambda_{b_1})\lambda_{b_0}$$
holds. Its inverse is given by
$$\varphi_e^{-1}:w_1\otimes w_2 \otimes \ldots \otimes w_e \to \theta^{e-1}(w_1)\theta^{e-2}(w_2)\ldots w_e.$$
\end{theorem}
\begin{proof}
$\varphi_e$ maps a basis to a basis, so is an isomorphism of abelian groups. Now it suffices to prove $\varphi_e^{-1}$ is multiplicative. Let $w_1,\ldots,w_e,w_1',\ldots,w_e' \in \Gamma_1$. We see
\begin{align*}
\varphi_e^{-1}(\otimes_{1 \leq i \leq e} w_i\cdot \otimes_{1 \leq i \leq e} w_i')= \varphi_e^{-1}(\otimes_{1 \leq i \leq e} w_iw_i')\\
=\theta^{e-1}(w_1w_1')\theta^{e-2}(w_2w_2')\ldots w_ew_e'\\
=\theta^{e-1}(w_1)\theta^{e-2}(w_2)\ldots w_e\cdot\theta^{e-1}(w_1')\theta^{e-2}(w_2')\ldots w_e'\\
=\varphi_e^{-1}(\otimes_{1 \leq i \leq e} w_i)\cdot \varphi_e^{-1}(\otimes_{1 \leq i \leq e} w_i').
\end{align*}
So we are done.
\end{proof}
\begin{corollary}\label[corollary]{2.9}
There is an induced isomorphism
$$\Gamma_e(\mathbb{C}) \to \Gamma_1(\mathbb{C})\otimes_\mathbb{C}\otimes \Gamma_1(\mathbb{C})\otimes \ldots \otimes_\mathbb{C}\Gamma_1(\mathbb{C})=\Gamma_1(\mathbb{C})^{\otimes e}.$$
We still denote this map by $\varphi_e$.
\end{corollary}
\begin{corollary}
For any $e$, $\Gamma_e(\mathbb{C})$ can be generated by a single element as a $\mathbb{C}$-algebra, and $\Gamma_e(\mathbb{Q})=\Gamma_e\otimes_\mathbb{Z}\mathbb{Q}$ can be generated by a single element as a $\mathbb{Q}$-algebra.    
\end{corollary}
\begin{proof}
We have $\Gamma_e(\mathbb{C})\cong \mathbb{C}^{p^e}$ as $\mathbb{C}$-algebra, and any element in $\Gamma_e(\mathbb{C})$ with pairwise distinct coordinates will generate $\Gamma_e(\mathbb{C})$ as a $\mathbb{C}$-algebra. Elements in $\Gamma_e(\mathbb{Q})$ such that at least two coordinates of $\varphi_e(v)$ coincide consist of a finite union of $\mathbb{Q}$-linear spaces which is not the full subspace. So we can always choose $v$ such that $\varphi_e(v)$ has pairwise distinct coordinates, and then it generates $\Gamma_e(\mathbb{Q})$ as a $\mathbb{Q}$-algebra. 
\end{proof}
\begin{remark}
The above statement is not true over $\mathbb{Z}$. Consider $p=2$ where $\Gamma_1=\mathbb{Z}[x]/(x^2-2x)$. Then $\Gamma_2=\mathbb{Z}[x,y]/(x^2-2x,y^2-2y)$ and $\Gamma_2\otimes_\mathbb{Z}\mathbb{F}_2=\mathbb{F}_2[x,y]/(x^2,y^2)$, which is not generated by a single element since any square of a non-unit is $0$ and any square of a unit is $1$. So $\Gamma_2$ is not generated by a single element.
\end{remark}
\subsection{Some functions on $\Gamma_e(\mathbb{C})$ and $\Gamma_1(\mathbb{C})^{\otimes e}$}
From now on, we will assume $p$ is an odd prime number.
\begin{definition}
\begin{enumerate}
\item $\mu: \Gamma \to \mathbb{C}$ is the linear map $\delta_i \mapsto 1$. Under this map, $\lambda_i \mapsto 0$ for $i \neq 0$ and $\lambda_0 \mapsto 1$. We will still denote its restriction on $\Gamma_e$ and the induced map on $\Gamma_e(\mathbb{C})$ by $\mu$ by abusing the notation. On the level of representations this represents the minimal number of generators of a representation.
\item $\ell: \Gamma \to \mathbb{C}$ is the linear map $\delta_i \mapsto i$. Under this map, $\lambda_i \mapsto (-1)^i$ for $i \geq 0$. We will still denote its restriction on $\Gamma_e$ and the induced map on $\Gamma_e(\mathbb{C})$ by $\ell$ by abusing the notation. On the level of representations this represents the dimension of a representation, which is also multiplicative.
\item for any implicitly determined $e$, $\tilde\mu=\mu^{\otimes e}: \Gamma_1(\mathbb{C})^{\otimes e} \to \mathbb{C}$.
\item for any implicitly determined $e$, $\tilde\ell=\ell^{\otimes e}: \Gamma_1(\mathbb{C})^{\otimes e} \to \mathbb{C}$.
\item We define $\Delta_e=\delta_p^{\otimes e} \in \Gamma_1(\mathbb{C})^{\otimes e}$.
\end{enumerate}
\end{definition}
\begin{proposition}
$\varphi_e(\delta_{p^e})=\Delta_e$.    
\end{proposition}
\begin{proof}
\begin{align*}
\varphi_e(\delta_{p^e})=\varphi_e(\sum_{0 \leq a<p^e}(-1)^a\lambda_a)=\varphi_e(\sum_{0 \leq a_{e-1},\ldots,a_1,a_0<p}(-1)^{\sum_{i}a_i}\lambda_{(a_{e-1}\ldots a_1a_0)_p})\\
=\sum_{0 \leq a_{e-1},\ldots,a_1,a_0<p}(-1)^{\sum_{i}a_i}\lambda_{a_{e-1}}\otimes\lambda_{a_{e-2}^{\sigma(a_{e-1})}}\otimes\ldots\otimes\lambda_{a_2^{\sigma(a_{e-1}\ldots a_3)}}\otimes\lambda_{a_1^{\sigma(a_{e-1}\ldots a_2)}}.
\end{align*} 
We observe that no matter what parity $a_{e-1}\ldots a_{j+1}$ is, $a_j^{\sigma(a_{e-1}\ldots a_{j+1})}$ has the same parity as $a_j$ and runs through $\{0,1,\ldots,p-1\}$ as $a_j$ does. So the above sum is equal to
$$\otimes_{0 \leq j \leq e-1}\sum_{0 \leq a_j<p}(-1)^{a_j}\lambda_{a_j}=\otimes_{0 \leq j \leq e-1}\delta_p=\Delta_e.$$
\end{proof}
\begin{proposition}
\begin{enumerate}
\item For any $v \in \Gamma_e(\mathbb{C})$, $v\delta_{p^e}=\ell(v)\delta_{p^e}$.
\item For any $w \in \Gamma_1(\mathbb{C})^{\otimes e}$, $w\Delta_e=\tilde\ell(w)\Delta_e$.
\item For any $v \in \Gamma_e(\mathbb{C})$, $\ell(v)=\tilde\ell(\varphi_e(v))$.
\item $\tilde\ell$ is multiplicative on $\Gamma_1(\mathbb{C})^{\otimes e}$.
\item For any $v \in \Gamma_e(\mathbb{C})$, $\mu(v)=\tilde\mu(\varphi_e(v))$.
\end{enumerate}
In this sense, we may say $\ell,\mu$ are compatible with the isomorphism $\varphi_e$.
\end{proposition}
\begin{proof}
(1): since $\delta_{p^e}$ is the unique projective $\mathbb{Z}/p^e\mathbb{Z}$-representation, and the tensor product of a projective representation with any representation is still projective, $v\delta_{p^e}$ is a multiple of $\delta_{p^e}$. The coefficient $\ell(v)$ comes from the comparison of dimensions of representations.

(2): this is true since $v\delta_p=\ell(v)\delta_p$ and $\tilde\ell$ is multilinear.

(3): by (1) and (2) we have for any $v \in \Gamma_e(\mathbb{C})$,
$$\ell(v)\Delta_e=\ell(v)\varphi_e(\delta_{p^e})=\varphi_e(\ell(v)\delta_{p^e})=\varphi_e(v\delta_{p^e})=\varphi_e(v)\Delta_e=\tilde\ell(\varphi_e(v))\Delta_e$$
and $\varphi_e(\delta_{p^e})=\Delta_e$ is not torsion, so $\ell(v)=\tilde\ell(\varphi_e(v))$.

(4): this is trivial by (3).

(5): it suffices to verify for the basis element $v=\delta_a,1 \leq a<p^e$, and $\varphi_e(\delta_a)=\lambda_{b_{e-1}}\otimes \lambda_{b_{e-2}}\otimes \ldots \otimes \lambda_{b_0}$. Both sides are equal to $1$ exactly when $a=1$ and all $b_i=0$, otherwise both sides are equal to $0$.
\end{proof}

\section{Hilbert-Kunz functions of quadrics}
As an application of \Cref{2.8} and \Cref{2.9}, we will compute the Hilbert-Kunz function of the Fermat quadratic hypersurface $Q_{p,d}=\mathbbm{k}[[x_1,\ldots,x_d]]/(x_1^2+\ldots+x_d^2)$, and express it explicitly in terms of elements in $\Gamma_1(\mathbb{C})$. We assume $p$ is odd from now on.
\begin{definition}
For $e \in \mathbb{Z}_+$, we define
$$v_e=\frac{1}{2}\delta_{\frac{p^e-1}{2}}+\frac{1}{2}\delta_{\frac{p^e+1}{2}} \in \Gamma_e(\mathbb{C}), w_e=\varphi_e(v_e),a_{p,e,d}=2^d\tilde\mu(w_e^d),b_{p,d}=\lim_{e \to \infty}\frac{a_{p,e,d}}{p^{e(d-1)}}.$$
\end{definition}
\begin{proposition}
The Hilbert-Kunz function of $\mathbbm{k}[[x_1,\ldots,x_d]]/(x_1^2+\ldots+x_d^2)$ is given by
$$e \mapsto a_{p,e,d}.$$
Therefore, its Hilbert-Kunz multiplicity is $b_{p,d}$.
\end{proposition}
\begin{proof}
This just comes from Han-Monsky's method; as a $\mathbbm{k}$-object with respect to $(x_1^2+\ldots+x_d^2)$, $\mathbbm{k}[[x_1,\ldots,x_d]]/(x_1^{p^e},\ldots,x_d^{p^e})=(2v_e)^d$, so 
$\mathbbm{k}[[x_1,\ldots,x_d]]/(x_1^{p^e},\ldots,x_d^{p^e},x_1^2+\ldots+x_d^2)=\mu(2v_e)^d=2^d\mu(v_e^d)=2^d\tilde\mu(w_e^d)$.
\end{proof}
Now we compute $w_e^d$. We denote $r=(p-1)/2$ and discuss according to parity of $r$.

\textbf{Case 1: $p \equiv 1 \operatorname{mod} 4$.} In this case $r$ is even.
\begin{proposition}\label[proposition]{3.3}
The image $w_e$ of $v_e$ in $\Gamma_1(\mathbb{C})^{\otimes e}$ is given by  
$$\varphi_e(v_e)=\delta_{r}\otimes \delta_p\otimes \ldots \otimes \delta_p+\lambda_r\otimes \delta_r\otimes \delta_p\otimes \ldots \otimes \delta_p+\ldots+\lambda_r\otimes\ldots\otimes \lambda_r\otimes(\frac{1}{2}\delta_r+\frac{1}{2}\delta_{r+1}).$$
\end{proposition}
\begin{proof}
We first claim the following equality holds:
$$\varphi_e(\delta_{\frac{p^e-1}{2}})=\delta_{r}\otimes \delta_p\otimes \ldots \otimes \delta_p+\lambda_r\otimes \delta_r\otimes \delta_p\otimes \ldots \otimes \delta_p+\ldots+\lambda_r\otimes\ldots\otimes \lambda_r\otimes\delta_r.$$  
We see the $p$-adic expansion of $\frac{p^e-1}{2}$ is $(rr\ldots r)_p$. By definition of $\delta_i$ and $\lambda_i$,
$$\delta_{(rr\ldots r)_p}=\sum_{a=(a_e\ldots a_2a_1)_p<(rr\ldots r)_p}(-1)^a\lambda_a.$$
Now we arrange them into $e$ groups according to digits of the indices of $\lambda$. The indices of the first group are $(00\ldots 0)_p\sim(r-1,p-1,\ldots,p-1)_p$, of the second group are $(r0\ldots 0)_p\sim(r,r-1,p-1,\ldots,p-1)_p$, and so on. The indices of the last group are $(rr\ldots r0)_p \sim (rr\ldots rr-1)_p$. We can use the same proof of $\varphi_e(\delta_{p^e})=\Delta_e$, and notice that the signs always match on the following equalities:
$$\varphi_e(\sum_{(00\ldots 0)_p\leq a=(a_e\ldots a_2a_1)_p<(r0\ldots 0)_p}(-1)^a\lambda_a)=\sum_{0\leq a<r}(-1)^a\lambda_a\otimes (\sum_{0\leq a<p}(-1)^a\lambda_a)^{e-1}=\delta_r\otimes \delta_p^{\otimes (e-1)},$$
$$\varphi_e(\sum_{(r0\ldots 0)_p\leq a=(a_e\ldots a_2a_1)_p<(rr\ldots 0)_p}(-1)^a\lambda_a)=\lambda_r\otimes\sum_{0\leq a<r}(-1)^a\lambda_a\otimes (\sum_{0\leq a<p}(-1)^a\lambda_a)^{e-2}=\lambda_r\otimes\delta_r\otimes \delta_p^{\otimes (e-2)},$$
$\ldots$
$$\varphi_e(\sum_{(rr\ldots r0)_p\leq a=(a_e\ldots a_2a_1)_p<(rr\ldots r)_p}(-1)^a\lambda_a)=\lambda_r^{\otimes (e-1)}\sum_{0\leq a<r}(-1)^a\lambda_a=\lambda_r^{\otimes (e-1)}\otimes\delta_r.$$
Adding them up leads to the equality. Similarly we have
$$\varphi_e(\delta_{\frac{p^e+1}{2}})=\delta_{r}\otimes \delta_p\otimes \ldots \otimes \delta_p+\lambda_r\otimes \delta_r\otimes \delta_p\otimes \ldots \otimes \delta_p+\ldots+\lambda_r\otimes\ldots\otimes \lambda_r\otimes\delta_{r+1}.$$
Therefore,
$$\varphi_e(v_e)=\delta_{r}\otimes \delta_p\otimes \ldots \otimes \delta_p+\lambda_r\otimes \delta_r\otimes \delta_p\otimes \ldots \otimes \delta_p+\ldots+\lambda_r\otimes\ldots\otimes \lambda_r\otimes(\frac{1}{2}\delta_r+\frac{1}{2}\delta_{r+1}).$$
\end{proof}
\begin{proposition}
\begin{enumerate}
\item $\tilde\ell(w_e)=\ell(v_e)=p^e/2$.
\item $w_1=\frac{1}{2}(\delta_r+\delta_{r+1})$, $w_e=\delta_r\otimes \Delta_{e-1}+\lambda_r\otimes w_{e-1}$.
\end{enumerate}    
\end{proposition}
\begin{proposition}
Let $d$ be a positive integer. The $d$-th power of $w_e$ satisfies the following recurrence relation:
\begin{align*}
w_e^d=\lambda_r^d\otimes w_{e-1}^d+\sum_{1 \leq i \leq d}{d \choose i}\frac{p^{e(d-1)}}{2^{d-i}p^{d-1}}\delta_r^i\lambda_r^{d-i}\otimes \Delta_{e-1}. 
\end{align*}
\end{proposition}
\begin{proof}
We see
\begin{align*}
w_e^d=(\delta_r\otimes\Delta_{e-1}+\lambda_r\otimes w_{e-1})^d\\
=\sum_{0 \leq i \leq d}{d \choose i}\delta_r^i\lambda_r^{d-i}\otimes \Delta_{e-1}^iw_{e-1}^{d-i}.
\end{align*}
When $i$ is nonzero, there is at least one $\Delta_{e-1}$ in the expression, and therefore
\begin{align*}
\Delta_{e-1}^iw_{e-1}^{d-i}=\tilde\ell(\Delta_{e-1}^{i-1}w_{e-1}^{d-i})\Delta_{e-1}\\
=\tilde\ell(\Delta_{e-1})^{i-1}\tilde\ell(w_{e-1})^{d-i}\Delta_{e-1}.
\end{align*}
But we have
$$\ell(\Delta_{e-1})=p^{e-1},\ell(w_{e-1})=p^{e-1}/2.$$
So
$$\Delta_{e-1}^iw_{e-1}^{d-i}=\frac{p^{e(d-1)}}{2^{d-i}p^{d-1}}\Delta_{e-1}.$$
So
\begin{align*}
w_e^d=\lambda_r^d\otimes w_{e-1}^d+\sum_{1 \leq i \leq d}{d \choose i}\frac{p^{e(d-1)}}{2^{d-i}p^{d-1}}\delta_r^i\lambda_r^{d-i}\otimes \Delta_{e-1}.
\end{align*}
\end{proof}
Now we compute $a_{p,e,d}=2^d\tilde\mu(w_e^d)$.
\begin{definition}
For $p$ prime and $d$ positive integer, we define
\begin{align*}
c_{p,d}=\mu(\sum_{1 \leq i \leq d}{d\choose i}\frac{1}{2^{-i}p^{d-1}}\delta_r^i\lambda_r^{d-i})=\frac{1}{p^{d-1}}\mu((2\delta_r+\lambda_r)^d-\lambda_r^d).
\end{align*}
\end{definition}
\begin{proposition}
We have
$$a_{p,e,d}=c_{p,d}p^{e(d-1)}+c_{p,d}\mu(\lambda_r^d)p^{(e-1)(d-1)}+\ldots+c_{p,d}\mu(\lambda_r^d)^{e-2}p^{2(d-1)}+2^d\mu(\lambda_r^d)^{e-1}\mu(w_1^d).$$
\end{proposition}
\begin{proof}
We have
\begin{align*}
w_e^d=\sum_{1 \leq i \leq d}{d \choose i}\frac{p^{e(d-1)}}{2^{d-i}p^{d-1}}\delta_r^i\lambda_r^{d-i}\otimes \Delta_{e-1}\\
+\lambda_r^d\otimes\sum_{1 \leq i \leq d}{d \choose i}\frac{p^{(e-1)(d-1)}}{2^{d-i}p^{d-1}}\delta_r^i\lambda_r^{d-i}\otimes \Delta_{e-2}\\
+\ldots+(\lambda_r^d)^{\otimes (e-2)}\otimes\sum_{1 \leq i \leq d}{d \choose i}\frac{p^{2(d-1)}}{2^{d-i}p^{d-1}}\delta_r^i\lambda_r^{d-i}\otimes \Delta_{1}+(\lambda_r^d)^{\otimes (e-1)}\otimes w_1^d.
\end{align*}
Now apply $2^d\tilde\mu$ on both sides. Note that $\tilde\mu=\mu^{\otimes e}$, $\mu$ is $\mathbb{C}$-linear, and $\tilde\mu(\Delta_i)=1$ for any $i$.
\end{proof}
\begin{lemma}
We have $0\leq\mu(\lambda_r^d)<p^{d-1}$.    
\end{lemma}
\begin{proof}
We consider the semigroup $\Gamma_1(\mathbb{C})^+=\sum_{0 \leq i<p}\mathbb{N}\lambda_i$ and the restriction of $\mathbb{Z}$-linear function $\mu^+: \Gamma_1(\mathbb{C})^+ \to \mathbb{N}$ such that $\mu^+(\lambda_i)=1$ for any $i$. By \Cref{2.3} we see $\Gamma_1(\mathbb{C})^+$ is closed under multiplication and for any $x,y \in \Gamma_1(\mathbb{C})^+,\mu^+(xy)\leq p\mu^+(x)\mu^+(y)$. In particular $\lambda_r^d \in \Gamma_1(\mathbb{C})^+$. Also restricted on $\Gamma_1(\mathbb{C})^+$, $0\leq \mu\leq \mu^+$ and $\mu(v)=\mu^+(v)$ if and only if $v$ is a multiple of $\lambda_0$. However, we see $\lambda_r$ has a nonzero $\lambda_r$-coefficient, and by the product rule, the $\lambda_r$-coefficient of $\lambda_r^d$ will never vanish. Therefore,
$$0\leq \mu(\lambda_r^d)< \mu^+(\lambda_r^d)\leq p^{d-1}\mu^+(\lambda_r)^d=p^{d-1}.$$
\end{proof}
\begin{proposition}
We have
$$\frac{a_{p,e,d}}{p^{e(d-1)}}=c_{p,d}\frac{1-\frac{\mu(\lambda_r^d)^{e-1}}{p^{(e-1)(d-1)}}}{1-\frac{\mu(\lambda_r^d)}{p^{d-1}}}+\frac{2^d\mu(\lambda_r^d)^{e-1}\mu(w_1^d)}{p^{e(d-1)}}$$
and
$$b_{p,d}=\lim_{e \to \infty}\frac{a_{p,e,d}}{p^{e(d-1)}}=c_{p,d}\frac{1}{1-\frac{\mu(\lambda_r^d)}{p^{d-1}}}=\frac{\mu((2\delta_r+\lambda_r)^d-\lambda_r^d)}{p^{d-1}-\mu(\lambda_r^d)}=\frac{\mu((\delta_{r+1}+\delta_r)^d-\lambda_r^d)}{p^{d-1}-\mu(\lambda_r^d)}.$$
\end{proposition}

\textbf{Case 2: $p \equiv 3 \operatorname{mod} 4$.} In this case $r$ is odd.

The map $\varphi_e$ will reflect the digit after one $r$. We denote
$$\delta_r^*=\sum_{0 \leq a<r}(-1)^a\lambda_{a^*}=\sum_{r<a\leq p-1}(-1)^a\lambda_a=\delta_p-\delta_{r+1},\delta_{r+1}^*=\delta_p-\delta_r$$
and define
$$\delta_a^{\sigma(n)}=\begin{cases}
\delta_a & n \textup{ even}\\
\delta_a^* & n \textup{ odd.}
\end{cases}$$
We also notice that $\lambda_{r^*}=\lambda_r$.
\begin{proposition}The image $w_e$ of $v_e$ in $\Gamma_1(\mathbb{C})^{\otimes e}$ is given by  
$$\varphi_e(v_e)=\delta_{r}\otimes \delta_p\otimes \ldots \otimes \delta_p+(-\lambda_r)\otimes \delta_r^*\otimes \delta_p\otimes \ldots \otimes \delta_p+\ldots+(-\lambda_r)\otimes\ldots\otimes (-\lambda_r)\otimes(\frac{1}{2}\delta_r+\frac{1}{2}\delta_{r+1})^{\sigma(e-1)}.$$
\end{proposition}
\begin{proof}
The proof is similar to \Cref{3.3}. We have
$$\varphi_e(\sum_{(00\ldots 0)_p\leq a=(a_e\ldots a_2a_1)_p<(r0\ldots 0)_p}(-1)^a\lambda_a)=\sum_{0\leq a<r}(-1)^a\lambda_a\otimes (\sum_{0\leq a<p}(-1)^a\lambda_a)^{e-1}=\delta_r\otimes \delta_p^{\otimes (e-1)},$$
\begin{align*}
\varphi_e(\sum_{(r0\ldots 0)_p\leq a=(a_e\ldots a_2a_1)_p<(rr\ldots 0)_p}(-1)^a\lambda_a)=(-\lambda_r)\otimes\sum_{0\leq a<r}(-1)^a\lambda_{a^*}\otimes (\sum_{0\leq a<p}(-1)^a\lambda_a)^{e-2}\\
=(-\lambda_r)\otimes\delta_r^*\otimes \delta_p^{\otimes (e-2)},    
\end{align*}
\begin{align*}
\varphi_e(\sum_{(rr0\ldots 0)_p\leq a=(a_e\ldots a_2a_1)_p<(rrr\ldots 0)_p}(-1)^a\lambda_a)\\
=(-\lambda_r)\otimes(-\lambda_r)\otimes\sum_{0\leq a<r}(-1)^a\lambda_{a}\otimes (\sum_{0\leq a<p}(-1)^a\lambda_a)^{e-3}\\
=(-\lambda_r)\otimes(-\lambda_r)\otimes\delta_r\otimes \delta_p^{\otimes (e-3)},    
\end{align*}
$\ldots$
$$\varphi_e(\sum_{(rr\ldots r0)_p\leq a=(a_e\ldots a_2a_1)_p<(rr\ldots r)_p}(-1)^a\lambda_a)=(-\lambda_r)^{\otimes (e-1)}\sum_{0\leq a<r}(-1)^a\lambda_{a^{\sigma(e-1)}}=(-\lambda_r)^{\otimes (e-1)}\otimes\delta_r^{\sigma(e-1)}.$$
Replacing $\delta_r$ with $\delta_r^*$ alternately and $\frac{1}{2}\delta_r+\frac{1}{2}\delta_{r+1}$ with $(\frac{1}{2}\delta_r+\frac{1}{2}\delta_{r+1})^{\sigma(e-1)}$, we get the result.
\end{proof}
Now we will set 
$$w_e^*=\delta_{r}^*\otimes \delta_p\otimes \ldots \otimes \delta_p+(-\lambda_r)\otimes \delta_r\otimes \delta_p\otimes \ldots \otimes \delta_p+\ldots+(-\lambda_r)\otimes\ldots\otimes (-\lambda_r)\otimes(\frac{1}{2}\delta_r+\frac{1}{2}\delta_{r+1})^{\sigma(e)}.$$
\begin{proposition}
\begin{enumerate}
\item $\tilde\ell(w_e)=\tilde\ell(w_e^*)=p^e/2$.
\item $w_1=\frac{1}{2}(\delta_r+\delta_{r+1})$, $w_1^*=\frac{1}{2}(\delta_r^*+\delta_{r+1}^*)=\frac{1}{2}(2\delta_p-\delta_r-\delta_{r+1})$, $w_e=\delta_r\otimes \Delta_{e-1}+(-\lambda_r)\otimes w_{e-1}^*$, $w_e^*=\delta_r^*\otimes \Delta_{e-1}+(-\lambda_r)\otimes w_{e-1}$.
\end{enumerate}    
\end{proposition}
\begin{proposition}
Let $d$ be a positive integer. The $d$-th power of $w_e$ satisfies the following recurrence relation:
\begin{align*}
w_e^d=(-\lambda_r)^d\otimes (w_{e-1}^*)^d+\sum_{1 \leq i \leq d}{d \choose i}\frac{p^{e(d-1)}}{2^{d-i}p^{d-1}}\delta_r^i(-\lambda_r)^{d-i}\otimes \Delta_{e-1},  
\end{align*}
\begin{align*}
(w_e^*)^d=(-\lambda_r)^d\otimes w_{e-1}^d+\sum_{1 \leq i \leq d}{d \choose i}\frac{p^{e(d-1)}}{2^{d-i}p^{d-1}}(\delta_r^*)^i(-\lambda_r)^{d-i}\otimes \Delta_{e-1}.
\end{align*}
\end{proposition}
\begin{definition}
For $p$ prime and $d$ positive integer, we define
\begin{align*}
c_{p,d}'=\mu(\sum_{1 \leq i \leq d}{d\choose i}\frac{1}{2^{-i}p^{d-1}}\delta_r^i(-\lambda_r)^{d-i})=\frac{1}{p^{d-1}}\mu((2\delta_r-\lambda_r)^d-(-\lambda_r)^d),\\  c_{p,d}'^*=\mu(\sum_{1 \leq i \leq d}{d\choose i}\frac{1}{2^{-i}p^{d-1}}(\delta_r^*)^i(-\lambda_r)^{d-i})=\frac{1}{p^{d-1}}\mu((2\delta_r^*-\lambda_r)^d-(-\lambda_r)^d).
\end{align*}
\end{definition}
Here we have $2\delta_r-\lambda_r=\delta_r+\delta_{r+1}$ and $2\delta_r^*-\lambda_r=2\delta_p-\delta_r-\delta_{r+1}$. Therefore,
\begin{align*}
(2\delta_p-\delta_r-\delta_{r+1})^d=c\delta_p+(-\delta_r-\delta_{r+1})^d    
\end{align*}
for some $c$. Applying $\ell$ on both sides, we get $c=(1-(-1)^d)p^{d-1}$. Therefore, if $d$ is even, then $(2\delta_r^*-\lambda_r)^d=(2\delta_r-\lambda_r)^d$ and $c'_{p,d}=c_{p,d}'^*$. If $d$ is odd, then $(2\delta_r^*-\lambda_r)^d=2p^{d-1}\delta_p-(2\delta_r-\lambda_r)^d$.
\begin{proposition}
We have
$$a_{p,e,d}=c'_{p,d}p^{e(d-1)}+c_{p,d}'^*\mu(-\lambda_r^d)p^{(e-1)(d-1)}+\ldots+c_{p,d}'^{\sigma(e-2)}\mu(-\lambda_r^d)^{e-2}p^{2(d-1)}+2^d\mu(-\lambda_r^d)^{e-1}\mu((w_1^{\sigma(e-1)})^d).$$
\end{proposition}
\begin{proposition}
Set $e_0=\lceil (e-1)/2 \rceil$ and $e_1=e-1-e_0$. We have
$$\frac{a_{p,e,d}}{p^{e(d-1)}}=c_{p,d}'\frac{1-\frac{\mu(-\lambda_r^d)^{2e_0}}{p^{2e_0(d-1)}}}{1-\frac{\mu(-\lambda_r^d)^2}{p^{2(d-1)}}}+c_{p,d}'^*\frac{\mu(-\lambda_r^d)}{p^{d-1}}\frac{1-\frac{\mu(-\lambda_r^d)^{2e_1}}{p^{2e_1(d-1)}}}{1-\frac{\mu(-\lambda_r^d)^2}{p^{2(d-1)}}}+\frac{2^d\mu(-\lambda_r^d)^{e-1}\mu((w_1^{\sigma(e-1)})^d)}{p^{e(d-1)}}$$
and
$$b_{p,d}=\lim_{e \to \infty}\frac{a_{p,e,d}}{p^{e(d-1)}}=c_{p,d}'\frac{1}{1-\frac{\mu(-\lambda_r^d)^2}{p^{2(d-1)}}}+c_{p,d}'^*\frac{\mu(-\lambda_r^d)}{p^{d-1}}\frac{1}{1-\frac{\mu(-\lambda_r^d)^2}{p^{2(d-1)}}}.$$
\end{proposition}
Now we discuss according to parity of $d$. If $d$ is even, then $c_{p,d}'=c_{p,d}'^*$ and
\begin{align*}
b_{p,d}=c_{p,d}'\frac{1}{1-\frac{\mu(-\lambda_r^d)}{p^{(d-1)}}}=\frac{\mu((2\delta_r-\lambda_r)^d-(-\lambda_r)^d)}{p^{d-1}-\mu(-\lambda_r^d)}=\frac{\mu((\delta_{r+1}+\delta_r)^d-(-\lambda_r)^d)}{p^{d-1}-\mu(-\lambda_r^d)}    .
\end{align*}
If $d$ is odd, then
\begin{align*}
b_{p,d}=c_{p,d}'\frac{1}{1-\frac{\mu(-\lambda_r^d)^2}{p^{2(d-1)}}}+c_{p,d}'^*\frac{\mu(-\lambda_r^d)}{p^{d-1}}\frac{1}{1-\frac{\mu(-\lambda_r^d)^2}{p^{2(d-1)}}}\\
=\frac{\mu((2\delta_r-\lambda_r)^d-(-\lambda_r)^d)}{p^{d-1}}\frac{1}{1-\frac{\mu(-\lambda_r^d)^2}{p^{2(d-1)}}}+\frac{\mu((2\delta_r^*-\lambda_r)^d-(-\lambda_r)^d)}{p^{d-1}}\frac{\mu(-\lambda_r^d)}{p^{d-1}}\frac{1}{1-\frac{\mu(-\lambda_r^d)^2}{p^{2(d-1)}}}\\
=\frac{\mu((2\delta_r-\lambda_r)^d-(-\lambda_r)^d)}{p^{d-1}}\frac{1}{1-\frac{\mu(-\lambda_r^d)^2}{p^{2(d-1)}}}\\
+\frac{\mu(2p^{d-1}\delta_p-(2\delta_r-\lambda_r)^d-(-\lambda_r)^d)}{p^{d-1}}\frac{\mu(-\lambda_r^d)}{p^{d-1}}\frac{1}{1-\frac{\mu(-\lambda_r^d)^2}{p^{2(d-1)}}}.
\end{align*}
Set $p^{d-1}=A,\mu(-\lambda_r^d)=B,\mu(2\delta_r-\lambda_r)=C$, the above is equal to
$$\frac{C-B}{A}\frac{1}{1-B^2/A^2}+\frac{2A-C-B}{A}\frac{B}{A}\frac{1}{1-B^2/A^2}=\frac{B+C}{A+B}.$$
Therefore,
$$b_{p,d}=\frac{\mu((2\delta_r-\lambda_r)^d+(-\lambda_r)^d)}{p^{d-1}+\mu(-\lambda_r^d)}=\frac{\mu((\delta_{r+1}+\delta_r)^d+(-\lambda_r)^d)}{p^{d-1}+\mu(-\lambda_r^d)}.$$

Now we summarize both cases $p \equiv 1 \operatorname{mod} 4$ and $p \equiv 3 \operatorname{mod} 4$, we get$$b_{p,d}=\frac{\mu((\delta_r+\delta_{r+1})^d)-(-1)^{rd}\mu((\delta_{r+1}-\delta_r)^d)}{p^{d-1}-(-1)^{rd}\mu((\delta_{r+1}-\delta_r)^d)}=1+\frac{\mu((\delta_r+\delta_{r+1})^d)-p^{d-1}}{p^{d-1}-(-1)^{rd}\mu((\delta_{r+1}-\delta_r)^d)},$$
which recovers \cite[Corollary 3.3]{PSSY}. Moreover, the computation actually gives the Hilbert-Kunz function, which has the form $b_{p,d}p^{e(d-1)}+C(e)\mu((-1)^r\lambda_r^d)^e$. Here $C(e)$ is a constant for $d$ even and is periodic of period 2 for $d$ odd. This $C(e)$ can be computed from products in $\Gamma_1$.

\section{Analytic description of $b_{p,d}$}
The main theorem in \cite{PSSY} expresses $b_{p,d}$ in terms of a fraction whose denominator and numerator involve Ehrhart polynomials. This expression leads to the fact that $b_{p,d}$ is a rational function in $p^2$ with rational coefficients. By analyzing the coefficients, we can prove that $p \mapsto b_{p,d}$ is decreasing for $p \gg 0$.

On the other hand, the re-verified previous result \cite[Corollary 3.3]{PSSY} allows us to compute $b_{p,d}$ using products in $\Gamma_1(\mathbb{C})$. Following this result, we will apply Gelfand transform to compute the power of an element. This method yields an analytic description of $b_{p,d}$, which allows us to derive its asymptotic behavior and monotonicity for all $p$ in the next section.
\subsection{Gelfand transform of $\Gamma_1(\mathbb{C})$ and action of $\mu$ on powers}
Following the notations of \cite{BensonBanach}, we will call a ring homomorphism $\Gamma_1 \to \mathbb{C}$ a species of $\Gamma_1$. They naturally extend to $\Gamma_1(\mathbb{C}) \to \mathbb{C}$ by extending the scalars. By \cite{Green} $\Gamma_1(\mathbb{C})$ is a semisimple $\mathbb{C}$-algebra, and it is finite-dimensional, so there are finitely many species. Therefore, the Chinese remainder theorem says that the Gelfand transform gives an isomorphism between the ring $\Gamma_1(\mathbb{C})$ and the ring of $\mathbb{C}$-valued functions on its maximal spectrum which is in one-to-one correspondence with the set of species. Here we do not discuss the $C^*$-algebra structure of $\Gamma_1(\mathbb{C})$.

\begin{theorem}[\cite{Green}]
The following list gives all species $\Gamma_1 \to \mathbb{C}$:
\begin{enumerate}
\item for $1 \leq k \leq p-1$, $s_k: \Gamma_1 \to \mathbb{C}$, $\delta_i \to \frac{\sin(ik\pi/p)}{\sin(k\pi/p)}$.
\item $s_p: \Gamma_1 \to \mathbb{C}$, $\delta_i \to i$.
\end{enumerate}
Let $s=(s_1,\ldots,s_p)$, then there is an isomorphism of rings
$$\Gamma_1(\mathbb{C}) \xrightarrow[]{s} \prod_{1 \leq i \leq p}\mathbb{C},$$
where the product on the right side is taken componentwisely.
\end{theorem}
\begin{lemma}
For any $w \in \Gamma_1(\mathbb{C})$, we have
$$\mu(w)=\frac{1}{p}(\sum_{1 \leq i \leq p-1}(1+\cos\frac{i\pi}{p})s_i(w)+s_p(w)).$$ 
Moreover for any positive integer $d$, we have
$$\mu(w^d)=\frac{1}{p}(\sum_{1 \leq i \leq p-1}(1+\cos\frac{i\pi}{p})s_i(w)^d+s_p(w)^d).$$    
\end{lemma}
\begin{proof}
The second statement is clear from the first since $s_i$'s are ring homomorphisms, so commute with taking $d$-th power. We note that $s_i \in \operatorname{Hom}_\mathbb{C}(\Gamma_1(\mathbb{C}),\mathbb{C}), 1 \leq i \leq p$ are linearly independent in $\operatorname{Hom}_\mathbb{C}(\Gamma_1(\mathbb{C}),\mathbb{C})$ because if $s_1,\ldots,s_p$ are linearly dependent, then the image of $s$ is a proper subspace of $\mathbb{C}^p$, which is a contradiction. So they form a $\mathbb{C}$-basis of $\operatorname{Hom}_\mathbb{C}(\Gamma_1(\mathbb{C}),\mathbb{C})$ and we can write $\mu=\sum_{1 \leq j \leq p}c_js_j:\Gamma_1(\mathbb{C}) \to \mathbb{C}$ for some $c_1,\ldots,c_p \in \mathbb{C}$. This means $\mu(\delta_i)=\sum_{1 \leq j \leq p}c_js_j(\delta_i)$ for any $i$. Now we solve for such $c_i$. Let
$$A=\begin{pmatrix}
s_1(\delta_1) & s_2(\delta_1) & \ldots & s_p(\delta_1)\\
s_1(\delta_2) & s_2(\delta_2) & \ldots & s_p(\delta_2)\\ 
\vdots & \vdots & \ddots & \vdots\\ 
s_1(\delta_p) & s_2(\delta_p) & \ldots & s_p(\delta_p)
\end{pmatrix}$$
In other words,
$$A=\begin{pmatrix}
\frac{\sin(\pi/p)}{\sin(\pi/p)} & \frac{\sin(2\pi/p)}{\sin(2\pi/p)} & \ldots & \frac{\sin((p-1)\pi/p)}{\sin((p-1)\pi/p)} & 1\\
\frac{\sin(2\pi/p)}{\sin(\pi/p)} & \frac{\sin(4\pi/p)}{\sin(2\pi/p)} & \ldots & \frac{\sin(2(p-1)\pi/p)}{\sin((p-1)\pi/p)} & 2\\ 
\vdots & \vdots & \ddots & \vdots & \vdots\\
\frac{\sin((p-1)\pi/p)}{\sin(\pi/p)} & \frac{\sin(2(p-1)\pi/p)}{\sin(2\pi/p)} & \ldots & \frac{\sin((p-1)^2\pi/p)}{\sin((p-1)\pi/p)} & p-1\\ 
0 & 0 & \ldots & 0 & p
\end{pmatrix}$$
Then
$$(1,1,\ldots,1)^T=A(c_1,\ldots,c_p)^T.$$
Note that this implies $c_p=1/p$. Therefore, let $B$ be the upper left submatrix of order $(p-1)$:
$$B=\begin{pmatrix}
\frac{\sin(\pi/p)}{\sin(\pi/p)} & \frac{\sin(2\pi/p)}{\sin(2\pi/p)} & \ldots & \frac{\sin((p-1)\pi/p)}{\sin((p-1)\pi/p)}\\
\frac{\sin(2\pi/p)}{\sin(\pi/p)} & \frac{\sin(4\pi/p)}{\sin(2\pi/p)} & \ldots & \frac{\sin(2(p-1)\pi/p)}{\sin((p-1)\pi/p)}\\ 
\vdots & \vdots & \ddots & \vdots &\\
\frac{\sin((p-1)\pi/p)}{\sin(\pi/p)} & \frac{\sin(2(p-1)\pi/p)}{\sin(2\pi/p)} & \ldots & \frac{\sin((p-1)^2\pi/p)}{\sin((p-1)\pi/p)}
\end{pmatrix}$$
we see 
$$(1-\frac{1}{p},1-\frac{2}{p},\ldots,\frac{1}{p})^T=B(c_1,\ldots,c_{p-1})^T$$
and
$$B^{-1}(1-\frac{1}{p},1-\frac{2}{p},\ldots,\frac{1}{p})^T=(c_1,\ldots,c_{p-1})^T.$$
Note that if $p,m$ are two positive integers, we can verify
$$\sum_{1 \leq i \leq p-1}\cos(\frac{im\pi}{p})=\sum_{1 \leq i \leq p-1}\operatorname{Re}(e^{\frac{im\pi\sqrt{-1}}{p}})=\begin{cases}
p-1 & 2p|m\\
-1 & 2|m,2p\nmid m\\
0 & 2 \nmid m.
\end{cases}$$
Now let $1 \leq j,k \leq p-1$ be two integers, and we have
$$\sin(\frac{ij\pi}{p})\sin(\frac{ik\pi}{p})=\frac{1}{2}(\cos(\frac{i(j-k)\pi}{p})-\cos(\frac{i(j+k)\pi}{p})).$$
We have: $0\leq |j-k|\leq p-1,2\leq j+k\leq 2p-2$. So if $j \neq k$, then $2p\nmid j-k,2p \nmid j+k$ and $j-k$ and $j+k$ have the same parity. Thus
$$\sum_{1 \leq i \leq p-1}\sin(\frac{ij\pi}{p})\sin(\frac{ik\pi}{p})=\begin{cases}
0 & j \neq k\\
\frac{1}{2}(p-1-(-1))=\frac{p}{2} & j=k.
\end{cases}$$
So we get
$$\tiny{B^{-1}=\frac{2}{p}\begin{pmatrix}
\sin(\pi/p)\sin(\pi/p) & \sin(2\pi/p)\sin(\pi/p) & \ldots & \sin((p-1)\pi/p)\sin(\pi/p)\\
\sin(2\pi/p)\sin(2\pi/p) & \sin(4\pi/p)\sin(2\pi/p) & \ldots & \sin(2(p-1)\pi/p)\sin(2\pi/p)\\ 
\vdots & \vdots & \ddots & \vdots &\\
\sin((p-1)\pi/p)\sin((p-1)\pi/p) & \sin(2(p-1)\pi/p)\sin((p-1)\pi/p) & \ldots & \sin((p-1)^2\pi/p)\sin((p-1)\pi/p)
\end{pmatrix}}$$
and for $1 \leq i \leq p-1$,
$$c_i=\frac{2}{p}\sin(\frac{i\pi}{p})\sum_{1 \leq j \leq p-1}\sin(\frac{ij\pi}{p})\frac{p-j}{p}.$$
We have for any $1 \leq i \leq p-1$,
$$\sum_{1 \leq j \leq p-1}\sin(\frac{ij\pi}{p})\frac{p-j}{p}=\frac{\sin(\frac{i\pi}{p})}{2-2\cos(\frac{i\pi}{p})},$$
so
$$c_i=\frac{1}{p}\frac{\sin^2(\frac{i\pi}{p})}{1-\cos(\frac{i\pi}{p})}=\frac{1+\cos(\frac{i\pi}{p})}{p}.$$
\end{proof}

\subsection{Computation of $b_{p,d}$}
We use Gelfand transform to compute
$$b_{p,d}=1+\frac{\mu((\delta_r+\delta_{r+1})^d)-p^{d-1}}{p^{d-1}-(-1)^{rd}\mu((\delta_{r+1}-\delta_r)^d)}.$$

We first compute the action of species on elements. Let $A_i$ be the $(r,i)$-th entry of $A$, which is
$$A_i=\frac{\sin\frac{ir\pi}{p}}{\sin\frac{i\pi}{p}}.$$
For $1 \leq i \leq p-1$, we have $\frac{i(r+1)\pi}{p}+\frac{ir\pi}{p}=i\pi$. So
$$s_i(\delta_r)=A_i,s_i(\delta_{r+1})=(-1)^{i+1}A_i.$$
We can also verify
$$A_{2k}=(-1)^{k+1}\frac{1}{2\cos\frac{k\pi}{p}},A_{2k+1}=(-1)^k\frac{1}{2\sin\frac{(2k+1)\pi}{2p}}.$$
In other words,
$$A_i=\begin{cases}
(-1)^{i/2+1}(2\cos(\frac{i\pi}{2p}))^{-1} & i \textup{ even}\\
(-1)^{(i-1)/2}(2\sin(\frac{i\pi}{2p}))^{-1} & i \textup{ odd.}
\end{cases}$$
Also,
$$s_p(\delta_r)=r,s_p(\delta_{r+1})=r+1.$$
Therefore, we have
\begin{align*}
b_{p,d}-1=\frac{\frac{1}{p}\left(\sum_{1 \leq i \leq p-1}(1+\cos\frac{i\pi}{p})((1+(-1)^{i+1})^d)A_i^d\right)}{p^{d-1}-(-1)^{rd}\frac{1}{p}\left(\sum_{1 \leq i \leq p-1}(1+\cos\frac{i\pi}{p})((-1)^{i+1}-1)^d)A_i^d+1\right)}\\
=\frac{\sum_{1 \leq i \leq p-1}(1+\cos\frac{i\pi}{p})((1+(-1)^{i+1})^d)A_i^d}{p^d-(-1)^{rd}\left(\sum_{1 \leq i \leq p-1}(1+\cos\frac{i\pi}{p})((-1)^{i+1}-1)^d)A_i^d+1\right)}\\
=\frac{\sum_{1 \leq i \leq p-1, i \textup{ odd}}(1+\cos\frac{i\pi}{p})(2A_i)^d}{p^d-(-1)^{rd}\left(\sum_{1 \leq i \leq p-1, i \textup{ even}}(1+\cos\frac{i\pi}{p})(-2A_i)^d+1\right)}\\
=\frac{\sum_{1 \leq i \leq p-1, i \textup{ odd}}(1+\cos\frac{i\pi}{p})(-1)^{d(i-1)/2}(\sin\frac{i\pi}{2p})^{-d}}{p^d-(-1)^{rd}\left(\sum_{1 \leq i \leq p-1, i \textup{ even}}(1+\cos\frac{i\pi}{p})(-1)^{di/2}(\cos\frac{i\pi}{2p})^{-d}+1\right)}\\
=\frac{\sum_{1 \leq i \leq p-1, i \textup{ odd}}(1+\cos\frac{i\pi}{p})(-1)^{d(i-1)/2}(\sin\frac{i\pi}{2p})^{-d}}{p^d-(-1)^{rd}\left(\sum_{1 \leq i \leq p-1, i \textup{ odd}}(1+\cos\frac{(p-i)\pi}{p})(-1)^{d(p-i)/2}(\cos\frac{(p-i)\pi}{2p})^{-d}+1\right)}\\
=\frac{\sum_{1 \leq i \leq p-1, i \textup{ odd}}(1+\cos\frac{i\pi}{p})(-1)^{d(i-1)/2}(\sin\frac{i\pi}{2p})^{-d}}{p^d-\sum_{1 \leq i \leq p-1, i \textup{ odd}}(1-\cos\frac{i\pi}{p})(-1)^{d(i-1)/2}(\sin\frac{i\pi}{2p})^{-d}-(-1)^{rd}}\\
=\frac{2\sum_{1 \leq i \leq p-1, i \textup{ odd}}(1-\sin^2\frac{i\pi}{2p})(-1)^{d(i-1)/2}(\sin\frac{i\pi}{2p})^{-d}}{p^d-2\sum_{1 \leq i \leq p-1, i \textup{ odd}}(-1)^{d(i-1)/2}(\sin\frac{i\pi}{2p})^{2-d}-(-1)^{rd}}.
\end{align*}
We set
$$S_{p,d}=\sum_{1 \leq i \leq p-1, i \textup{ odd}}\sin(\frac{i\pi}{2p})^{-d},T_{p,d}=\sum_{1 \leq i \leq p-1, i \textup{ odd}}(-1)^{(i-1)/2}\sin(\frac{i\pi}{2p})^{-d}.$$
Then,
$$b_{p,d}-1=\begin{cases}
\frac{2S_{p,d}-2S_{p,d-2}}{p^d-2S_{p,d-2}-1} & d \textup{ even}\\
\frac{2T_{p,d}-2T_{p,d-2}}{p^d-2T_{p,d-2}-(-1)^{r}} & d \textup{ odd}.
\end{cases}$$
This gives the analytic expression of $b_{p,d}$. We can extend its definition to all odd number $p \geq 3$, which is not necessarily prime.

\begin{remark}
There is also an alternative closed formula for $b_{p,d}-1$ in \cite[Section 4]{CSR26} which uses $\sec$ instead of $\sin$. Through some simple trigonometric substitutions, it can be shown that these two formulas are equivalent. The method in \cite{CSR26} is quite different from this paper in that it directly diagonalizes the multiplication by $\lambda_i$ as a linear map on $\Gamma_1(\mathbb{C})$ from computations in linear algebra.  
\end{remark}
\subsection{A rational expression of $b_{p,d}-1$ in terms of $p$}
The results in previous subsection is a closed formula for $b_{p,d}-1$. However, unlike the results in \cite{PSSY}, it is not direct to see from this result that for fixed $d$, $p \to b_{p,d}-1$ is a rational function. Now we derive another closed formula which is in the form of a rational function of $p$. At the same time, the proof for this formula also yields explicit formula for coefficients of two Ehrhart polynomials $F_d$ and $E_d$ studied in \cite{PSSY}. The proof relies on the Cauchy residue theorem, which relates the sums of residues to a contour integral. This formula can be quickly computed for small $d$.
\begin{theorem}\label{4.4}
Let
$$S_{p,d}=\sum_{1 \leq i \leq p-1, i \textup{ odd}}\sin(\frac{i\pi}{2p})^{-d},T_{p,d}=\sum_{1 \leq i \leq p-1, i \textup{ odd}}(-1)^{(i-1)/2}\sin(\frac{i\pi}{2p})^{-d}.$$
Suppose $d \geq 1$. Then:
\begin{enumerate}
\item When $d$ is even,
$$S_{p,d}=\frac{p}{2}\operatorname{Res}_{z=0}\frac{\tan(pz)}{\sin^d(z)}-\frac{1}{2}.$$
\item When $d$ is odd,
$$T_{p,d}=\frac{p}{2}\operatorname{Res}_{z=0}\frac{\sec(pz)}{\sin^d(z)}-\frac{1}{2}(-1)^r.$$
\end{enumerate}
\end{theorem}
\begin{proof}
Suppose $d$ is even. We denote
$$F_1(z)=\frac{\tan(pz)}{\sin^d(z)}.$$
Let $0<\epsilon<\frac{\pi}{2p}$ be a real number. Let $C$ be the contour which is the union of four segments: $\pi-\epsilon-iy \to \pi-\epsilon+iy \to -\epsilon+iy \to -\epsilon-iy \to \pi-\epsilon-iy$. It is a boundary of a rectangle. Denote the four segments starting from $\pi-\epsilon-iy$ by $C_1\sim C_4$ respectively. We will consider the limit of the contour integrals
$$\lim_{y\to\infty}\oint_C F_1(z).$$
Here is an illustration of the contour.
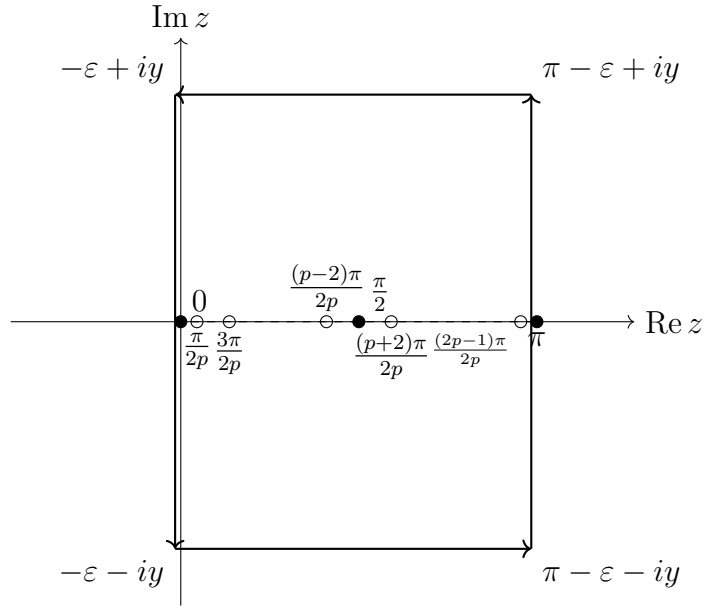
\begin{center}
\begin{tikzpicture}[scale=1.5]

\draw[->] (-1.5,0) -- (4,0) node[right] {$\operatorname{Re} z$};
\draw[->] (0,-2.5) -- (0,2.5) node[above] {$\operatorname{Im} z$};

\def\y{2}
\def\eps{0.05}

\coordinate (A) at ({pi-\eps},-\y);   
\coordinate (B) at ({pi-\eps},\y);    
\coordinate (C) at (-\eps,\y);       
\coordinate (D) at (-\eps,-\y);      

\draw[->,thick] (A) -- (B);
\draw[->,thick] (B) -- (C);
\draw[->,thick] (C) -- (D);
\draw[->,thick] (D) -- (A);

\node[above right] at (B) {$\pi-\varepsilon + iy$};
\node[above left] at (C) {$-\varepsilon + iy$};
\node[below left] at (D) {$-\varepsilon - iy$};
\node[below right] at (A) {$\pi-\varepsilon - iy$};

\filldraw (0,0) circle (1.5pt) node[above right] {$0$};
\filldraw ({pi},0) circle (1.5pt) node[below] {$\pi$};
\filldraw ({pi/2},0) circle (1.5pt) node[above right] {$\frac{\pi}{2}$};

\draw ({pi/22},0) circle (1.5pt) node[below] {$\frac{\pi}{2p}$};
\draw ({3*pi/22},0) circle (1.5pt) node[below] {$\frac{3\pi}{2p}$};
\draw ({9*pi/22},0) circle (1.5pt) node[above] {$\frac{(p-2)\pi}{2p}$};
\draw ({13*pi/22},0) circle (1.5pt) node[below] {$\frac{(p+2)\pi}{2p}$};
\draw ({21*pi/22},0) circle (1.5pt) node[below left] {\tiny$\frac{(2p-1)\pi}{2p}$};

\draw[dashed] (-\eps,0) -- ({pi-\eps},0);

\node[align=left] at (2,-3) {\footnotesize Illustration of the contour. Sums of residues of $F_1$ ($F_2$ respectively) at the \\ \footnotesize hollow circle left to $\pi/2$ and right to $\pi/2$ are both $-\frac{1}{p}S_{p,d}$ ($-\frac{1}{p}T_{p,d}$ respectively). \footnotesize ($\varepsilon \ll 1$, $y \gg 1$)};

\end{tikzpicture}
\end{center}

We see when $d$ is even, $F_1(z)$ is periodic of period $\pi$. Thus the integrals along $C_1$ and $C_3$ cancel. When $y \to \pm\infty$ and $x$ is bounded, $F_1(x+iy) \to 0$. So the integrals along $C_2,C_4$ go to $0$, therefore,
$$\lim_{y\to\infty}\oint_C F_1(z)=0.$$
Now we check the poles and residues of $F_1$: the poles inside $C$ are exactly $0$ and $\frac{i\pi}{2p}$ where $1 \leq i \leq 2p-1$ is odd. We see
$$\operatorname{Res}_{z=\frac{i\pi}{2p}}F_1(z)=-\frac{1}{p\sin^d(\frac{i\pi}{2p})}.$$
In particular, the residue at $\frac{\pi}{2}$ is $-\frac{1}{p}$. Also, $\sin^d(\frac{i\pi}{2p})=\sin^d(\frac{(2p-i)\pi}{2p})$. Now the residue theorem gives
$$\operatorname{Res}_{z=0}F_1(z)-\frac{1}{p}(2S_{p,d}+1)=0.$$
Therefore,
$$S_{p,d}=\frac{p}{2}\operatorname{Res}_{z=0}\frac{\tan(pz)}{\sin^d(z)}-\frac{1}{2}.$$
Now suppose $d$ is odd. To compute $T_{p,d}$, we consider
$$F_2(z)=\frac{\sec(pz)}{\sin^d(z)}.$$
Similarly we can prove
$$\lim_{y\to\infty}\oint_C F_2(z)=0,\operatorname{Res}_{z=\frac{i\pi}{2p}}F_2(z)=-\frac{(-1)^{(i-1)/2}}{p\sin^d(\frac{i\pi}{2p})}$$
so the residue theorem gives
$$T_{p,d}=\frac{p}{2}\operatorname{Res}_{z=0}\frac{\sec(pz)}{\sin^d(z)}-\frac{1}{2}(-1)^r.$$
\end{proof}
\begin{remark}
When $d=0$, we get
$$S_{p,0}=\frac{p-1}{2}\neq -\frac{1}{2}$$
so the above formula does not apply. The reason for this is that we don't have 
$$F_1(x+iy)\to 0$$
in this case, so the contour integeral does not converge to $0$ anymore.
\end{remark}
\begin{definition}
Define
$$\rho_d(z)=\sin(z)^{-d}=z^{-d}(z/\sin z)^d.$$
It is a meromorphic function on $\mathbb{C}$ whose Laurent power series expansion converges for $z \in (0,\pi/2)$. Let $$z^d\rho_d(z)=\rho_{d,0}+\rho_{d,1}z+\rho_{d,2}z^2+\ldots$$
be the Taylor expansion of $z^d\rho_d(z)$ near $0$.
\end{definition}
Since $z^d\rho_d(z)$ is an even function, $\rho_{d,i}=0$ when $i$ is odd. Recall that the Taylor expansion of $\tan$ and $\sec$ are given by:
$$
\tan z = \sum_{n=1}^\infty \frac{(-1)^{n-1} 2^{2n} (2^{2n} - 1) B_{2n}}{(2n)!} z^{2n-1}
$$
$$
\sec z = \sum_{n=0}^\infty \frac{(-1)^n E_{2n}}{(2n)!} z^{2n}
$$
where $B_n$ is the Bernoulli number and $E_n$ is the Euler number defined by $\frac{x}{e^x - 1} = \sum_{n=0}^{\infty} B_n \frac{x^n}{n!}$ and $\frac{1}{\cosh x} = \sum_{n=0}^{\infty} E_n \frac{x^n}{n!}$.

\begin{theorem}\label{4.7}
Let $d \geq 1$ be an integer. We have:
\begin{enumerate}
\item When $d$ is even, $$S_{p,d}=\frac{1}{2}\sum_{i=1}^{d/2}\frac{(-1)^{i-1} 2^{2i} (2^{2i} - 1) B_{2i}\rho_{d,d-2i}p^{2i}}{(2i)!}-\frac{1}{2}.$$
\item When $d$ is odd, $$T_{p,d}=\frac{1}{2}\sum_{i=0}^{(d-1)/2}\frac{(-1)^i E_{2i}\rho_{d,d-2i-1}p^{2i+1}}{(2i)!}-\frac{1}{2}(-1)^r.$$
\item When $d \geq 4$ is even, 
$$b_{p,d}-1=\frac{\sum_{i=1}^{d/2}\frac{(-1)^{i-1} 2^{2i} (2^{2i} - 1) B_{2i}\rho_{d,d-2i}p^{2i}}{(2i)!}-\sum_{i=1}^{(d-2)/2}\frac{(-1)^{i-1} 2^{2i} (2^{2i} - 1) B_{2i}\rho_{d-2,d-2i-2}p^{2i}}{(2i)!}}{p^d-\sum_{i=1}^{(d-2)/2}\frac{(-1)^{i-1} 2^{2i} (2^{2i} - 1) B_{2i}\rho_{d-2,d-2i-2}p^{2i}}{(2i)!}}$$
\item When $d \geq 3$ is odd,
$$b_{p,d}-1=\frac{\sum_{i=0}^{(d-1)/2}\frac{(-1)^i E_{2i}\rho_{d,d-2i-1}p^{2i+1}}{(2i)!}-\sum_{i=0}^{(d-3)/2}\frac{(-1)^i E_{2i}\rho_{d-2,d-2i-3}p^{2i+1}}{(2i)!}}{p^d-\sum_{i=0}^{(d-3)/2}\frac{(-1)^i E_{2i}\rho_{d-2,d-2i-3}p^{2i+1}}{(2i)!}}$$
\end{enumerate}
\end{theorem}
\begin{proof}
We can directly extract the residue as the coefficient of $z^{-1}$ in the Taylor expansion. We get: when $d$ is even,
$$\operatorname{Res}_{z=0}\frac{\tan(pz)}{\sin^d(z)}=\sum_{i=1}^{d/2}\frac{(-1)^{i-1} 2^{2i} (2^{2i} - 1) B_{2i}\rho_{d,d-2i}p^{2i-1}}{(2i)!}$$
and when $d$ is odd,
$$\operatorname{Res}_{z=0}\frac{\sec(pz)}{\sin^d(z)}=\sum_{i=0}^{(d-1)/2}\frac{(-1)^i E_{2i}\rho_{d,d-2i-1}p^{2i}}{(2i)!}$$
and the rest is straightforward computation.
\end{proof}
We recall that the $d$-dimensional Fibonacci polytope $F_d$ is the subset of $[0,1]^d$ satisfying $x_i+x_{i+1}\leq 1,1 \leq i \leq d-1$ and $d$-dimensional extended Fibonacci polytope $E_d$ is the subset of $[-1,1]^d$ satisfying $|x_i|+|x_{i+1}|\leq 1,1 \leq i \leq d-1$.  
\begin{corollary}\label{4.8}
For $d \geq 1$, let $F_d(n)$ and $E_d(n)$ be the Ehrhart polynomials of the $d$-dimensional Fibonacci and extended Fibonacci polytopes. Then we have:
\begin{align*}
F_{d}(\frac{p-3}{2})=\begin{cases}
\frac{1}{2^{d}p}(\sum_{i=1}^{(d+1)/2}\frac{(-1)^{i-1} 2^{2i} (2^{2i} - 1) B_{2i}\rho_{d+1,d-2i+1}p^{2i}}{(2i)!}\\-\sum_{i=1}^{(d-1)/2}\frac{(-1)^{i-1} 2^{2i} (2^{2i} - 1) B_{2i}\rho_{d-1,d-2i-1}p^{2i}}{(2i)!}) & d \textup{ odd}\\
\\
\frac{1}{2^{d}p}(\sum_{i=0}^{d/2}\frac{(-1)^i E_{2i}\rho_{d+1,d-2i}p^{2i+1}}{(2i)!}\\-\sum_{i=0}^{(d-2)/2}\frac{(-1)^i E_{2i}\rho_{d-1,d-2i-2}p^{2i+1}}{(2i)!}) & d \textup{ even}.
\end{cases}
\end{align*}
for $d \geq 2$ and
$$E_{d}(\frac{p-1}{2})=\begin{cases}
\frac{1}{p}(\sum_{i=1}^{(d+1)/2}\frac{(-1)^{i-1} 2^{2i} (2^{2i} - 1) B_{2i}\rho_{d+1,d-2i+1}p^{2i}}{(2i)!}) & d \textup{ odd}\\
\frac{1}{p}(\sum_{i=0}^{d/2}\frac{(-1)^i E_{2i}\rho_{d+1,d-2i}p^{2i+1}}{(2i)!}) & d \textup{ even}.
\end{cases}
$$
for $d \geq 1$.
\end{corollary}
\begin{proof}
From the proof of \cite[Theorem 3.9]{PSSY} we get
$$\mu((\delta_r+\delta_{r+1})^d)-p^{d-1}=2^{d-1}F_{d-1}(\frac{p-3}{2})$$
and
$$p^{d-1}-(-1)^{rd}\mu((\delta_{r+1}-\delta_r)^d)=p^{d-1}-E_{d-3}(\frac{p-1}{2})$$
The computation in subsection 4.2 gives
$$\mu((\delta_r+\delta_{r+1})^d)-p^{d-1}=\begin{cases}
\frac{1}{p}(2S_{p,d}-2S_{p,d-2}) & d \textup{ even}\\
\frac{1}{p}(2T_{p,d}-2T_{p,d-2}) & d \textup{ odd}.
\end{cases}$$
and
$$p^{d-1}-(-1)^{rd}\mu((\delta_{r+1}-\delta_r)^d)=\begin{cases}
\frac{1}{p}(p^d-2S_{p,d-2}-1) & d \textup{ even}\\
\frac{1}{p}(p^d-2T_{p,d-2}-(-1)^r) & d \textup{ odd}.
\end{cases}$$
and for $d \geq 1$, $S_{p,d},T_{p,d}$ are computed in \Cref{4.4}. Combining all these yields
\begin{align*}
F_{d-1}(\frac{p-3}{2})=\begin{cases}
\frac{1}{2^{d-1}p}(\sum_{i=1}^{d/2}\frac{(-1)^{i-1} 2^{2i} (2^{2i} - 1) B_{2i}\rho_{d,d-2i}p^{2i}}{(2i)!}\\-\sum_{i=1}^{(d-2)/2}\frac{(-1)^{i-1} 2^{2i} (2^{2i} - 1) B_{2i}\rho_{d-2,d-2i-2}p^{2i}}{(2i)!}) & d \textup{ even}\\
\\
\frac{1}{2^{d-1}p}(\sum_{i=0}^{(d-1)/2}\frac{(-1)^i E_{2i}\rho_{d,d-2i-1}p^{2i+1}}{(2i)!}\\-\sum_{i=0}^{(d-3)/2}\frac{(-1)^i E_{2i}\rho_{d-2,d-2i-3}p^{2i+1}}{(2i)!}) & d \textup{ odd}.
\end{cases}
\end{align*}
for $d-2 \geq 1$ and
$$E_{d-3}(\frac{p-1}{2})=\begin{cases}
\frac{1}{p}(\sum_{i=1}^{(d-2)/2}\frac{(-1)^{i-1} 2^{2i} (2^{2i} - 1) B_{2i}\rho_{d-2,d-2i-2}p^{2i}}{(2i)!}) & d \textup{ even}\\
\frac{1}{p}(\sum_{i=0}^{(d-3)/2}\frac{(-1)^i E_{2i}\rho_{d-2,d-2i-3}p^{2i+1}}{(2i)!}) & d \textup{ odd}.
\end{cases}
$$
for $d-2 \geq 1$. Replacing $d$ by $d+1$ in the expression of $F_{d-1}$ and replacing $d$ by $d+3$ in the expression of $E_{d-3}$, we get the result.
\end{proof}
\begin{remark}
Another expression of $F_d(n)$ is shown in \cite[Theorem 1.5]{KAH}.    
\end{remark}
\section{Asymptotic behavior and monotonicity of $b_{p,d}-1$}
In this section we will confirm Yoshida's conjecture:
\begin{theorem}\label{5.1}
For any $d$,
$$p \mapsto b_{p,d}-1$$
is decreasing on the set of odd numbers at least $3$. In particular, it is decreasing on the set of all odd prime numbers. 
\end{theorem}
To prove \Cref{5.1}, we need to analyze the asymptotic behavior of $S_{p,d}$ and $T_{p,d}$. We first show that if we replace $S_{p,d}$ or $T_{p,d}$ with an approximation, then the corresponding expression is decreasing in $p$, and then we show the exact value is sufficiently close to the approximation, so it does not affect the decreasing property. The form of the approximation will vary according to whether $p$ is large or $d$ is large.

We first check some cases where $p$ or $d$ is small.
\subsection{Small values of $p$ and $d$}
\begin{example}
We directly compute $b_{p,d}-1$ for some small $d$.  We see 
\begin{align*}
S_{p,0}=\frac{p-1}{2}, T_{p,1}=\frac{p-(-1)^r}{2},S_{p,2}=\frac{p^2-1}{2},\\
T_{p,3}=\frac{p^3+p-2(-1)^r}{4},S_{p,4}=\frac{p^4+2p^2-3}{6},\\
T_{p,5}=\frac{5p^5+10p^3+9p-24(-1)^r}{48}.   
\end{align*}
They can be computed from the Laurent expansion of functions $F_1,F_2$ in subsection 4.3. This gives
$$b_{p,d}-1=1,\frac{1}{2},\frac{1}{3},\frac{5p^2+3}{24p^2+12}$$
for $d=2,3,4,5$ respectively. This recovers the result in \cite{WYconj05}. In these cases $p \mapsto b_{p,d}-1$ is decreasing. For $d=1$ $b_{p,d}=2$ by definition. 
\end{example}
\begin{example}
We see
$$b_{3,d}-1=2\frac{2^d-2^{d-2}}{3^d-2\cdot 2^{d-2}-(-1)^d}=\frac{3\cdot 2^{d-1}}{3^d-2^{d-1}-(-1)^d}$$
$$b_{5,d}-1=2\frac{\sin\frac{\pi}{10}^{(-d)}+(-\sin\frac{3\pi}{10})^{(-d)}-\sin^{2-d}\frac{\pi}{10}-(-\sin\frac{3\pi}{10})^{2-d}}{5^d-2\sin^{2-d}\frac{\pi}{10}-2(-\sin\frac{3\pi}{10})^{2-d}-1}.$$
Here
$$\sin\frac{\pi}{10}=\frac{\sqrt{5}-1}{4},\sin\frac{3\pi}{10}=\frac{\sqrt{5}+1}{4}.$$
Now we prove $b_{3,d}-1>b_{5,d}-1$. We set
$$p=(\sin\frac{\pi}{10})^{-1}=\sqrt{5}+1,q=(\sin\frac{3\pi}{10})^{-1}=\sqrt{5}-1.$$
We see
$$b_{3,d}-1\geq \frac{3}{2}(\frac{2}{3})^{d}$$
and
$$b_{5,d}-1=\frac{2(p^d+(-1)^dq^d-p^{d-2}-(-1)^dq^{d-2})}{5^d-2p^{d-2}-2(-1)^dq^{d-2}-1}.$$
We have $\frac{2p^{d-2}+2q^{d-2}+1}{5^d}\leq \frac{2p^4+2q^4+1}{5^6}<0.015$ when $d \geq 6$. Then
$$b_{5,d}-1\leq \frac{2(p^2-1)p^{d-2}+2(q^2-1)q^{d-2}}{5^d(1-0.015)}.$$
We see when $d=7$,
$$\frac{2(p^2-1)p^{d-2}+2(q^2-1)q^{d-2}}{5^d(1-0.015)}=0.0874...<0.0877...=(\frac{3}{2})(\frac{2}{3})^d$$
and $\frac{p}{5},\frac{q}{5}<\frac{2}{3}$, so the above inequality also holds for $d \geq 7$, and $b_{3,d}-1>b_{5,d}-1$. The $d=6$ case can be computed directly:
$$b_{3,6}-1=0.137...>0.135...=b_{5,6}-1.$$
Also we have computed $d=2,3,4,5$ as above. In sum,
$$b_{3,d}-1\geq b_{5,d}-1$$
for all $d$ and the inequality is strict when $d \geq 5$.
\end{example}
\begin{remark}
In \cite{PSSY}, it is mentioned that after concrete computation, $p \mapsto b_{p,d}$ is decreasing in terms of $p$ for $d \leq 31$. We show their computational result of $b_{p,d}-1$ for $6 \leq d \leq 15$ here: 
\end{remark} 

$$b_{p,6}-1 = \frac{2(1+p^2)}{5(2+3p^2)},b_{p,7}-1 = \frac{45+86p^2+61p^4}{270+570p^2+720p^4}$$

$$b_{p,8}-1 = \frac{24+31p^2+17p^4}{168+273p^2+315p^4},b_{p,9}-1 = \frac{1575+3163p^2+3093p^4+1385p^6}{56(225+484p^2+659p^4+720p^6)}$$

$$b_{p,10}-1 = \frac{2(72+103p^2+82p^4+31p^6)}{9(144+242p^2+298p^4+315p^6)}$$

$$b_{p,11}-1 = \frac{99225+204444p^2+228694p^4+154396p^6+50521p^8}{90(11025+23941p^2+33811p^4+38935p^6+40320p^8)}$$

$$b_{p,12}-1 = \frac{2(2880+4354p^2+4079p^4+2396p^6+691p^8)}{55(1152+1972p^2+2518p^4+2773p^6+2835p^8)}$$

$$b_{p,13}-1 = \frac{9823275+20559681p^2+24778126p^4+20126754p^6+10482999p^8+2702765p^{10}}{132(893025+1950246p^2+2814296p^4+3349754p^6+3578279p^8+3628800p^{10})}$$

$$b_{p,14}-1 = \frac{4(43200+67614p^2+69473p^4+50025p^6+23427p^8+5461p^{10})}{39(57600+99752p^2+130332p^4+147723p^6+154543p^8+155925p^{10})}$$

$$b_{p,15}-1 = \mathsmaller{\frac{1404728325+2971290258p^2+3755076443p^4+3394389116p^6+2201503851p^8+937215826p^{10}+199360981p^{12}}{182(108056025+236872791p^2+346708946p^4+422040774p^6+461849829p^8+476298835p^{10}+479001600p^{12})}}$$

By taking derivative with respect to $p$ we can verify the decreasing property in these cases. So, we have verified the case $p=3$ and the case $d \leq 15$. Therefore, we may assume $p \geq 5$ and $d \geq 16$ for the rest part of this section. We will prove $p \mapsto b_{p,d}$ is decreasing in this case in the following subsections.

\subsection{Asymptotic behavior of $S_{p,d}$}
In this subsection, we derive the asymptotic behavior of $S_{p,d}$ with respect to $p$ and $d$. This behavior will indicate the exact term in $S_{p,d}$ to take as its approximation.

Let $\rho(z)$ be a meromorphic function with a pole of order $d\geq 2$ near $0$. Write
$$\rho(z)=c_{-d}z^{-d}+c_{1-d}z^{1-d}+c_{2-d}z^{2-d}+\ldots.$$
Then we would expect, after proving the absolute convergence of double series,
$$\sum_{1 \leq i \leq p-1,i \operatorname{odd}}\rho(\frac{i}{p})=p^dc_{-d}\sum_{1 \leq i \leq p-1,i \operatorname{odd}}i^{-d}+p^{d-1}c_{1-d}\sum_{1 \leq i \leq p-1,i \operatorname{odd}}i^{1-d}+\ldots +p^2c_{-2}\sum_{1 \leq i \leq p-1,i \operatorname{odd}}i^{-2}+O(p)$$
and
\begin{align*}
\sum_{1 \leq i \leq p-1,i \operatorname{odd}}(-1)^{(i-1)/2}\rho(\frac{i}{p})=p^dc_{-d}\sum_{1 \leq i \leq p-1,i \operatorname{odd}}(-1)^{(i-1)/2}i^{-d}\\+p^{d-1}c_{1-d}\sum_{1 \leq i \leq p-1,i \operatorname{odd}}(-1)^{(i-1)/2}i^{1-d}+\ldots
+p^2c_{-2}\sum_{1 \leq i \leq p-1,i \operatorname{odd}}(-1)^{(i-1)/2}i^{-2}+O(p).   
\end{align*}
To study the above sum, we recall the definition of the Dirichlet $\lambda$-function and Dirichlet $\beta$-function.
\begin{align*}
\lambda(d)=\sum_{i \geq 1, i \textup{ odd}}i^{-d},\beta(d)=\sum_{i \geq 1, i \textup{ odd}}(-1)^{(i-1)/2}i^{-d}.    
\end{align*}
Denote their partial sums by
\begin{align*}
\lambda_p(d)=\sum_{1 \leq i<p, i \textup{ odd}}i^{-d},\beta_p(d)=\sum_{1\leq i<p, i \textup{ odd}}(-1)^{(i-1)/2}i^{-d}.    
\end{align*}
Using Riemann sums over the function $f(x)=x^{-d}$, it is straightforward to check when $p \geq 3$ and $d \geq 2$,
\begin{align*}
0<(-1)^{(p-1)/2}(\beta(d)-\beta_p(d))\leq \lambda(d)-\lambda_p(d)\\
\leq p^{-d}+\frac{1}{2}\int_p^{\infty}x^{-d}dx=p^{-d}+\frac{p}{2(d-1)}p^{-d}<p^{d-1}.    
\end{align*}
We have also defined the Taylor coefficients $\rho_{d,i}$ such that 
$$z^d/\sin^d(z)=\rho_{d,0}+\rho_{d,1}z+\rho_{d,2}z^2+\ldots$$
near $z=0$ in subsection 4.3.
\begin{lemma}
$\rho_{d,i}\geq 0$ for any $i$. Moreover, we have
\begin{enumerate}
\item for each $i$, $\rho_{d,i}$ is polynomial in $d$;
\item $\rho_{d,i}=0$ if $i$ is odd.
\item $\rho_{d,i}$ is a polynomial of degree $i/2$ if $i$ is even.
\end{enumerate}
\end{lemma}
\begin{proof}
By Euler's sine product formula,
$$z^d\rho_d(z)=\prod_{n \in \mathbb{Z}_{\geq 1}}(1-\frac{z^2}{n^2\pi^2})^{-d}=\prod_{n \in \mathbb{Z}_{\geq 1}}(\sum_{m_n\geq 0}{m_n+d-1\choose m_n}(\frac{z^2}{n^2\pi^2})^{m_n}).$$
So all the coefficients $\rho_{d,i}$ are nonnegative. Moreover, $\sin(z)/z=1-\frac{1}{6}z^2+\frac{1}{120}z^4+\ldots$ and $(1-z)^{-d}=\sum_{n \geq 0}{n+d-1\choose n}z^n$. So
$$z^d\rho_d(z)=1+\sum_{n \geq 1}{n+d-1\choose n}(\frac{1}{6}z^2-\frac{1}{120}z^4+\ldots)^n.$$
So we are done.
\end{proof}
We compute the first three nonvanishing terms of $z^d\rho_d(z)$:
\begin{align*}
z^d\rho_d(z)=(1-\frac{1}{6}z^2+\frac{1}{120}z^4+O(z^6))^{-d}\\
=1+d(\frac{1}{6}z^2-\frac{1}{120}z^4)+\frac{d(d+1)}{2}(\frac{1}{6}z^2-\frac{1}{120}z^4)^2+O(z^6)\\
=1+\frac{d}{6}z^2+\frac{5d^2+2d}{360}z^4+O(z^6).    
\end{align*}
That is,
$$\rho_{d,0}=1,\rho_{d,2}=\frac{d}{6},\rho_{d,4}=\frac{5d^2+2d}{360}.$$
The positivity of the coefficient implies that we can express $S_{p,d}$ as a convergent double sum of terms as below. All terms will be positive and their sum is finite. This allows us to change the order of summation.

\begin{center}
\begin{tabular}{|c||c|c|c|c|}
\hline
The sum $S_{p,d}$ & $\textcolor{red}{\sin(\frac{\pi}{2p})^{-d}}$ & $\sin(\frac{3\pi}{2p})^{-d}$ & $\sin(\frac{5\pi}{2p})^{-d}$ & $\ldots$ \\
\hline\hline
$\boldsymbol{\rho_{d,0}(\frac{\pi}{2})^{-d}\lambda_p(d)p^d}$ & $\textcolor{red}{\boldsymbol{\rho_{d,0}(\frac{\pi}{2p})^{-d}}}$ & $\boldsymbol{\rho_{d,0}(\frac{3\pi}{2p})^{-d}}$ & $\boldsymbol{\rho_{d,0}(\frac{5\pi}{2p})^{-d}}$ & $\ldots$ \\
\hline
$\rho_{d,2}(\frac{\pi}{2})^{2-d}\lambda_p(d-2)p^{d-2}$ & $\textcolor{red}{\rho_{d,2}(\frac{\pi}{2p})^{2-d}}$ & $\rho_{d,2}(\frac{3\pi}{2p})^{2-d}$ & $\rho_{d,2}(\frac{5\pi}{2p})^{2-d}$ & $\ldots$ \\
\hline
$\rho_{d,4}(\frac{\pi}{2})^{4-d}\lambda_p(d-4)p^{d-4}$ & $\textcolor{red}{\rho_{d,4}(\frac{\pi}{2p})^{4-d}}$ & $\rho_{d,4}(\frac{3\pi}{2p})^{4-d}$ & $\rho_{d,4}(\frac{5\pi}{2p})^{4-d}$ & $\ldots$ \\
\hline
$\vdots$ & $\vdots$ & $\vdots$ & $\vdots$ & $\ddots$ \\
\hline
\end{tabular}  
\end{center}

\textbf{Method of estimation:} The scales of different terms vary for different choices of $p,d$:
\begin{enumerate}
\item when $d$ is large compared with $p$, the dominant part is given by the red terms (the second column), which is approximately $\alpha^d$. 
\item when $p$ is large compared with $d$, the dominant part is given by the bolded terms (the second row), which is approximately $Cp^d$.
\end{enumerate}
Therefore, we need to divide the proof into two cases. We choose the following division:
\begin{enumerate}
\item when $p\leq 2^{d/4}$, we say this is the \textbf{large $d$ case}. We will approximate the sum using the first column, which is a power of sines.
\item when $p \geq 2^{d/4}$, we say this is the \textbf{large $p$ case}. We will approximate the sum using first several rows, which are monomials of $p$.
\end{enumerate}
We will treat these two cases separately. The estimation of $T_{p,d}$ is completely the same; see the following table.

\begin{center}
\begin{tabular}{|c||c|c|c|c|}
\hline
The sum $T_{p,d}$ & $\textcolor{red}{\sin(\frac{\pi}{2p})^{-d}}$ & $-\sin(\frac{3\pi}{2p})^{-d}$ & $\sin(\frac{5\pi}{2p})^{-d}$ & $\ldots$ \\
\hline\hline
$\boldsymbol{\rho_{d,0}(\frac{\pi}{2})^{-d}\beta_p(d)p^d}$ & $\textcolor{red}{\boldsymbol{\rho_{d,0}(\frac{\pi}{2p})^{-d}}}$ & $\boldsymbol{-\rho_{d,0}(\frac{3\pi}{2p})^{-d}}$ & $\boldsymbol{\rho_{d,0}(\frac{5\pi}{2p})^{-d}}$ & $\ldots$ \\
\hline
$\rho_{d,2}(\frac{\pi}{2})^{2-d}\beta_p(d-2)p^{d-2}$ & $\textcolor{red}{\rho_{d,2}(\frac{\pi}{2p})^{2-d}}$ & $-\rho_{d,2}(\frac{3\pi}{2p})^{2-d}$ & $\rho_{d,2}(\frac{5\pi}{2p})^{2-d}$ & $\ldots$ \\
\hline
$\rho_{d,4}(\frac{\pi}{2})^{4-d}\beta_p(d-4)p^{d-4}$ & $\textcolor{red}{\rho_{d,4}(\frac{\pi}{2p})^{4-d}}$ & $-\rho_{d,4}(\frac{3\pi}{2p})^{4-d}$ & $\rho_{d,4}(\frac{5\pi}{2p})^{4-d}$ & $\ldots$ \\
\hline
$\vdots$ & $\vdots$ & $\vdots$ & $\vdots$ & $\ddots$ \\
\hline
\end{tabular}  
\end{center}

Actually the terms appearing in the double sum formula for $T_{p,d}$ are exactly the same terms for $S_{p,d}$ with sign changes; this means the sum is still absolutely convergent, so we can change the order of summation in $T_{p,d}$ as well.
\subsection{The large $d$ case}
In the large $d$ case, we approximate $S_{p,d}$ or $T_{p,d}$ using $\sin(\frac{\pi}{2p})^{-d}$. The sum of the other terms is approximately $p\sin(\frac{3\pi}{2p})^{-d}\sim p\sin(\frac{\pi}{2p})^{-d}3^{-d}$, which is much smaller than $\sin(\frac{\pi}{2p})^{-d}$. We will first prove that the approximation of $b_{p,d}-1$ decreases,
and then we show the approximation is sufficiently accurate to capture the decreasing property of the original value $b_{p,d}-1$. We first introduce an approximation lemma for inequalities.
\begin{lemma}[Approximation of inequality between fractions]
Let $A,B,C,D$ be 4 real numbers, $\tilde{A},\tilde{B},\tilde{C},\tilde{D}$ be another 4 real numbers which are the approximations. Suppose we have
$$0<\frac{\tilde{C}}{\tilde{D}}<\frac{\tilde{A}}{\tilde{B}}.$$
Assume $\epsilon$ is a real number such that
$$\frac{\frac{\tilde{A}}{\tilde{B}}-\frac{\tilde{C}}{\tilde{D}}}{\frac{\tilde{C}}{\tilde{D}}}\geq\epsilon>0.$$
Suppose we have
$$\max\{\frac{|A-\tilde{A}|}{|\tilde{A}|},\frac{|B-\tilde{B}|}{|\tilde{B}|},\frac{|C-\tilde{C}|}{|\tilde{C}|},\frac{|D-\tilde{D}|}{|\tilde{D}|}\}\leq \min\{\frac{1}{8}\epsilon,\frac{1}{8}\}.$$
Then
$$\frac{C}{D}<\frac{A}{B}.$$
\end{lemma}
\begin{proof}
The above result is the consequence of two inequalities: when $\epsilon\leq 1$,
$$\frac{(1-\frac{1}{8}\epsilon)^2}{(1+\frac{1}{8}\epsilon)^2}(1+\epsilon)>1$$
and when $\epsilon \geq 1$,
$$\frac{(1-\frac{1}{8})^2}{(1+\frac{1}{8})^2}2>1.$$
\end{proof}
Now we prove the inequality for approximations. Set
$$\tilde{b}_{p,d}-1=\frac{2\sin(\frac{\pi}{2p})^{-d}-2\sin(\frac{\pi}{2p})^{2-d}}{p^d-2\sin(\frac{\pi}{2p})^{2-d}}.$$
If we denote $\alpha_p=\frac{\pi}{2p}$, then $\tilde{b}_{p,d}-1=2F(\alpha_p)$ for the real function
$$F(\alpha)=\frac{\alpha^d\sin(\alpha)^{-d}-\alpha^d\sin(\alpha)^{2-d}}{(\frac{\pi}{2})^d-2\alpha^d\sin(\alpha)^{2-d}}.$$
It is easy to see $F(\alpha)>0$ for $0<\alpha\leq\frac{\pi}{10}$. Let $a=\alpha$, $s=\sin\alpha$, $c=\cos\alpha$, and assume $0<\alpha\leq \frac{\pi}{10}$ from now on. Then
$$F(\alpha)=\frac{a^ds^{-d}c^2}{(\frac{\pi}{2})^d-2a^ds^{2-d}},$$
$$F'(\alpha)=\frac{a^{d-1}cs^{-1-d}[(\frac{\pi}{2})^d(dcs-2as^2-dac^2)+4a^{d+1}s^{2-d}]}{((\frac{\pi}{2})^d-2a^ds^{2-d})^2}.$$
Now we find a lower bound of $F'(\alpha)$. We observe that
$$z \mapsto z/\sin(z)$$
is increasing on $(0,\pi/2)$ and $0<\alpha<\frac{\pi}{10}$, so
$$1<a/s\leq (\frac{\pi}{10})/(\sin\frac{\pi}{10})=\frac{(\sqrt{5}+1)\pi}{10}<1.02$$
and for any $d\geq 16$
$$\frac{1}{2}(\frac{\pi}{2})^d<(\frac{\pi}{2})^d-2a^ds^{2-d}<(\frac{\pi}{2})^d,4a^{d+1}s^{2-d}\geq 0,$$
$$dcs-2as^2-dac^2=dc^2(\tan\alpha-\alpha)-2as^2\geq dc^2\frac{a^3}{3}-2a^3=2a^3(\frac{d}{6}c^2-1).$$
In sum, we replace $a/s$ with $1$ or $1.02$ when appropriate, throw away the positive term $4a^{d+1}s^{2-d}$, and use $c\geq \frac{\sqrt{3}}{2}$ since $\alpha<\frac{\pi}{6}$, we get
$$F'(\alpha)\geq (\frac{\pi}{2})^{-d}c2a(\frac{d}{6}c^2-1)\geq (\frac{\pi}{2})^{-d}\sqrt{3}a(\frac{d}{8}-1).$$
Therefore,
\begin{align*}
F(\alpha_p)-F(\alpha_{p+2})\geq (\frac{\pi}{2})^{-d}\sqrt{3}\alpha_{p+2}(\frac{d}{8}-1)(\alpha_p-\alpha_{p+2})\\
\geq (\frac{\pi}{2})^{-d}\sqrt{3}(\frac{\pi}{2})^2(\frac{d}{8}-1)\frac{2}{p(p+2)^2}.  
\end{align*}
We have
$$F(\alpha_{p+2})\leq \frac{2\cdot1.02^d}{(\pi/2)^d},$$
so
\begin{align*}
\frac{F(\alpha_p)-F(\alpha_{p+2})}{F(\alpha_{p+2})}\geq \epsilon:=1.02^{-d}\sqrt{3}(\frac{\pi}{2})^2(\frac{d}{8}-1)\frac{1}{p(p+2)^2}.
\end{align*}
Now we will be proving a sequence of inequalities which lead to \Cref{5.1} in the large $d$ case. We list all the inequalities we will prove here:
\begin{align}
\frac{|S_{p,d}-\sin(\frac{\pi}{2p})^{-d}|}{\sin(\frac{\pi}{2p})^{-d}-\sin(\frac{\pi}{2p})^{2-d}}<\frac{1}{16}\epsilon,\frac{|S_{p,d}-\sin(\frac{\pi}{2p})^{-d}|}{\sin(\frac{\pi}{2p})^{-d}-\sin(\frac{\pi}{2p})^{2-d}}<\frac{1}{16}, \tag{A1 \& A2}\\
\frac{|S_{p,d-2}-\sin(\frac{\pi}{2p})^{2-d}|}{\sin(\frac{\pi}{2p})^{-d}-\sin(\frac{\pi}{2p})^{2-d}}<\frac{1}{16}\epsilon,\frac{|S_{p,d-2}-\sin(\frac{\pi}{2p})^{2-d}|}{\sin(\frac{\pi}{2p})^{-d}-\sin(\frac{\pi}{2p})^{2-d}}<\frac{1}{16}, \tag{A3 \& A4}\\
\frac{|2S_{p,d-2}-2\sin(\frac{\pi}{2p})^{2-d}+1|}{p^d-2\sin(\frac{\pi}{2p})^{2-d}}<\frac{1}{8}\epsilon,\frac{|2S_{p,d-2}-2\sin(\frac{\pi}{2p})^{2-d}+1|}{p^d-2\sin(\frac{\pi}{2p})^{2-d}}<\frac{1}{8}. \tag{A5 \& A6}
\end{align}
\begin{align}
2<\frac{1}{16}2^{-d}(\frac{3+\sqrt{5}}{2\cdot 1.02})^{d}\frac{\sqrt{5}+5}{8}\sqrt{3}(\frac{\pi}{2})^2(\frac{d}{8}-1), \tag{C1}\\
2^{d/4}(\frac{3+\sqrt{5}}{2})^{-d}\frac{2(5-\sqrt{5})}{5}<\frac{1}{16}, \tag{C2}\\
4\frac{9\pi^2}{4}<2^{-\frac{d}{2}}(\frac{3\pi}{2})^{d}(\frac{3\pi(\sqrt{5}-1)}{10})^{-d}\frac{1}{8}1.02^{-d}\sqrt{3}(\frac{\pi}{2})^2(\frac{d}{8}-1), \tag{C3}\\
2\frac{9\pi^2}{4\cdot 5}(\frac{3\pi}{2})^{-d}(\frac{3\pi(\sqrt{5}-1)}{10})^{d}<\frac{1}{8}. \tag{C4}
\end{align}

Before we prove the above inequalities, we establish some other inequalities independent of (A1)-(A6) and (C1)-(C4). We observe that the following equalities hold:
\begin{align}
|S_{p,d-2}-\sin(\frac{\pi}{2p})^{2-d}|\leq |S_{p,d}-\sin(\frac{\pi}{2p})^{-d}|\leq p\cdot \sin(\frac{3\pi}{2p})^{-d}, \tag{B1}   
\end{align}
which is true since the difference consists of less than $p$ terms no greater than $\sin(\frac{3\pi}{2p})^{-d}$. Using triple-angle formula for sines, we have
\begin{align*}
\frac{p\sin(\frac{3\pi}{2p})^{-d}}{\sin(\frac{\pi}{2p})^{-d}-\sin(\frac{\pi}{2p})^{2-d}}=p\cdot (3-4\sin^2(\frac{\pi}{2p}))^{-d}\cos(\frac{\pi}{2p})^{-2}\\
\leq p\cdot (3-4\sin^2(\frac{\pi}{10}))^{-d}\cos(\frac{\pi}{10})^{-2}=p(\frac{3+\sqrt{5}}{2})^{-d}\frac{2(5-\sqrt{5})}{5}. \tag{B2}
\end{align*}
(B1) and (B2) imply that the left side of (A1) $\sim$ (A4) is at most $p(\frac{3+\sqrt{5}}{2})^{-d}\frac{2(5-\sqrt{5})}{5}$.
We also have $\sin(\frac{\pi}{2}x)\geq x$ for $x \in [0,1]$, therefore,
\begin{align}
p^d-2\sin(\frac{\pi}{2p})^{2-d}\geq p^d-2p^{d-2}=(1-\frac{2}{p^2})p^d\geq \frac{1}{2}p^d \tag{B3} 
\end{align}
and $S_{p,d-2}-\sin(\frac{\pi}{2p})^{2-d}$ is the sum of $(p-3)/2$ terms between $1$ and $\sin(\frac{3\pi}{2p})^{2-d}$, so
\begin{align*}
|2S_{p,d-2}-2\sin(\frac{\pi}{2p})^{2-d}+1|\leq p\sin(\frac{3\pi}{2p})^{2-d}
\leq p(\frac{3\pi}{2p})^2(\frac{3\pi}{2p})^{-d}\frac{(\frac{3\pi}{2p})^{d}}{(\sin\frac{3\pi}{2p})^{d}}\\
\leq p(\frac{3\pi}{2p})^2(\frac{3\pi}{2p})^{-d}\frac{(\frac{3\pi}{10})^{d}}{(\sin\frac{3\pi}{10})^{d}}=\frac{9\pi^2}{4p}(\frac{3\pi}{2p})^{-d}(\frac{3\pi(\sqrt{5}-1)}{10})^{d}. \tag{B4}
\end{align*}
(B3) and (B4) imply that the left side of (A5) $\sim$ (A6) is at most $2\frac{9\pi^2}{4p}(\frac{3\pi}{2})^{-d}(\frac{3\pi(\sqrt{5}-1)}{10})^{d}$.

Proof of (A1-A6) implies $b_{p,d}-1>b_{p+2,d}-1$ in the large $d$ case: this is the approximation lemma applied to $A/B=\frac{1}{2}(b_{p,d}-1)$, $C/D=\frac{1}{2}(b_{p+2,d}-1)$, $\tilde{A}/\tilde{B}=F(\alpha_p)=\frac{1}{2}(\tilde{b}_{p,d}-1)$ and $\tilde{C}/\tilde{D}=F(\alpha_{p+2})=\frac{1}{2}(\tilde{b}_{p+2,d}-1)$.

Proof of (C1) implies (A1), (A3): Suppose (C1) holds. We rewrite (C1) as
$$2\cdot 2^d<\frac{1}{16}(\frac{3+\sqrt{5}}{2\cdot 1.02})^{d}\frac{\sqrt{5}+5}{8}\sqrt{3}(\frac{\pi}{2})^2(\frac{d}{8}-1).$$
By the large $d$ condition that $p\leq 2^{d/4}$. If $p \geq 5$, then $(p+2)^2< 2p^2$. So
$$p^2(p+2)^2<2p^4\leq2\cdot 2^d<\frac{1}{16}(\frac{3+\sqrt{5}}{2\cdot 1.02})^{d}\frac{\sqrt{5}+5}{8}\sqrt{3}(\frac{\pi}{2})^2(\frac{d}{8}-1).$$
So after exchanging terms on both sides, we get
$$p(\frac{3+\sqrt{5}}{2})^{-d}\frac{2(5-\sqrt{5})}{5}<\frac{1}{16}1.02^{-d}\sqrt{3}(\frac{\pi}{2})^2(\frac{d}{8}-1)\frac{1}{p(p+2)^2}.$$
Now (A1) and (A3) are consequences of the above inequality plus (B1) and (B2).

Proof of (C2) implies (A2), (A4): suppose (C2) holds. Then by the large $d$ assumption $p \leq 2^{d/4}$ we see
$$p(\frac{3+\sqrt{5}}{2})^{-d}\frac{2(5-\sqrt{5})}{5}<\frac{1}{16}.$$
Now (A2) and (A4) are consequences of the above inequality plus (B1) and (B2).

Proof of (C3) implies (A5): we rewrite (C3) as
$$4\frac{9\pi^22^{\frac{d}{2}}}{4}(\frac{3\pi}{2})^{-d}(\frac{3\pi(\sqrt{5}-1)}{10})^{d}<\frac{1}{8}1.02^{-d}\sqrt{3}(\frac{\pi}{2})^2(\frac{d}{8}-1).$$
Using $(p+2)^2<2p^2<2\cdot 2^{d/2}$, we see (C3) implies
$$2\frac{9\pi^2(p+2)^2}{4}(\frac{3\pi}{2})^{-d}(\frac{3\pi(\sqrt{5}-1)}{10})^{d}<\frac{1}{8}1.02^{-d}\sqrt{3}(\frac{\pi}{2})^2(\frac{d}{8}-1),$$
which can be rewritten as
$$2\frac{9\pi^2}{4p}(\frac{3\pi}{2})^{-d}(\frac{3\pi(\sqrt{5}-1)}{10})^{d}<\frac{1}{8}1.02^{-d}\sqrt{3}(\frac{\pi}{2})^2(\frac{d}{8}-1)\frac{1}{p(p+2)^2}.$$
This, together with (B3) and (B4), implies (A5).

Proof of (C4) implies (A6): $p \geq 5$ and (C4) imply
$$2\frac{9\pi^2}{4p}(\frac{3\pi}{2})^{-d}(\frac{3\pi(\sqrt{5}-1)}{10})^{d}<\frac{1}{8}.$$
This, together with (B3) and (B4), implies (A6).

Proof of (C1)-(C4): we have
\begin{align*}
\frac{1}{2}\frac{3+\sqrt{5}}{2}>1,\sqrt[4]{2}\frac{2}{3+\sqrt{5}}<1,\\
\frac{1}{\sqrt{2}}\frac{3\pi}{2}\frac{10}{3\pi(\sqrt{5}-1)}\frac{1}{1.02}>1,\frac{2}{3\pi}\frac{3\pi(\sqrt{5}-1)}{10}<1.   
\end{align*}
Therefore, the right-hand sides of (C1) and (C3) are increasing in $d$, and the left-hand sides of (C2) and (C4) are decreasing in $d$. Therefore, it suffices to verify the inequalities when $d=16$ by straightforward computation. When $d=16$, these inequalities are approximately
$$2<13.08,3.63\cdot10^{-6}<\frac{1}{16},4\frac{9\pi^2}{4}<7.81\cdot10^6,1.73\cdot10^{-9}<\frac{1}{8},$$
which are true.

The estimate for $T_{p,d}$ is completely the same as $S_{p,d}$ since 
$$\max\{|T_{p,d-2}-\sin(\frac{\pi}{2p})^{2-d}|, |T_{p,d}-\sin(\frac{\pi}{2p})^{-d}|\}\leq \frac{p-3}{2}\cdot \sin(\frac{3\pi}{2p})^{-d}$$
still holds. Therefore, the large $d$ case is proved.

\subsection{The large $p$ case}
In this subsection we assume $d \geq 16$ and $p \geq 2^{d/4}$ is an odd number. In particular, $p \geq 16$. In the large $p$ case, the sum $S_{p,d}$ is approximately a polynomial of $p$. We will prove \Cref{5.1} using positivity of coefficients of the polynomial.

We have seen that the coefficients of the Taylor expansion of $z^d\rho_d(z)$ are always positive. This positivity leads to a fine estimate of $z^d\rho_d(z)$ up to a certain order.
\begin{proposition}
Let $D$ be any positive integer and $P_{d,D}(z)=\rho_{d,0}+\rho_{d,1}z+\ldots+\rho_{d,D-1}z^{D-1}$ be the Taylor expansion of $z^d\rho_d(z)$ up to $z^{D-1}$. Then for $z \in (0,\pi/2)$ we have
\begin{align*}
P_{d,D}(z)\leq z^d\rho_d(z)\leq P_{d,D}(z)+\frac{z^D}{(\pi/2)^D}(\rho_d(\frac{\pi}{2})-P_{d,D}(\frac{\pi}{2}))\\
=P_{d,D}(z)+((\frac{\pi}{2})^{-D}\rho_d(\frac{\pi}{2})-(\frac{\pi}{2})^{-D}P_{d,D}(\frac{\pi}{2}))z^D\leq P_{d,D}(z)+(\frac{\pi}{2})^{-D}\rho_d(\frac{\pi}{2})z^D.    
\end{align*}
In particular, when $D=d-1$, we have
\begin{align*}
P_{d,d-1}(z) \leq z^d\rho_d(z)
\leq P_{d,d-1}(z)+\frac{\pi}{2}z^{d-1}.
\end{align*}
\end{proposition}
\begin{proof}
The first line is just a reformulation of the inequality below for $z \in (0,\pi/2)$, which comes from the positivity of the coefficients in the Taylor expansion:
\begin{align*}
\rho_{d,0}+\rho_{d,1}z+\ldots+\rho_{d,D}z^D\\
\leq \rho_{d,0}+\rho_{d,1}z+\ldots+\rho_{d,D}z^D+\rho_{d,D+1}z^{D+1}+\rho_{d,D+2}z^{D+2}+\ldots\\
\leq \rho_{d,0}+\rho_{d,1}z+\ldots+\rho_{d,D}z^D+\rho_{d,D+1}z^D(\frac{\pi}{2})+\rho_{d,D+2}z^D(\frac{\pi}{2})^2+\ldots.
\end{align*}
The second line is straightforward as $P_{d,D}(\frac{\pi}{2})\geq 0$.
\end{proof}
\begin{lemma}\label[lemma]{5.9}
We have $\rho_{d,2i} \leq (\frac{d}{4})^i$ for $2i \leq d$.    
\end{lemma}
\begin{proof}
By Euler's sine product formula,  
\begin{align*}
z^d\rho_d(z)=\prod_{n \in \mathbb{Z}_{\geq 1}}(1-\frac{z^2}{n^2\pi^2})^{-d}\\
=\prod_{n \in \mathbb{Z}_{\geq 1}}(\sum_{m_n \geq 0}{m_n+d-1\choose m_n}(\frac{z^2}{n^2\pi^2})^{m_n}).
\end{align*}
We analyze the coefficient of $(\frac{z^2}{\pi^2})^i, 2i \leq d$ here, which is given by the partitions of $i$, that is, unordered sets $m_1,\ldots,m_k\geq 1$ such that $i=m_1+\ldots+m_k$. The coefficient will be:
$$\sum_{n_1,\ldots,n_k \in \mathbb{Z}_{\geq 1}\textup{ pairwise distinct}}\prod_{1 \leq j \leq k,\sum_j m_j=i} {m_j+d-1\choose m_j}\frac{1}{n_j^{2m_j}}.$$
We can throw away the denominator in ${m_j+d-1\choose m_j}=\frac{m_j(m_j+1)\ldots(m_j+d-1)}{m_j!}$ and replace all factors in the numerator with $3d/2$ since $m_j<d/2$, so
\begin{align*}
\sum_{n_1,\ldots,n_k \in \mathbb{Z}_{\geq 1}\textup{ pairwise distinct}}\prod_{1 \leq j \leq k,\sum_j m_j=i} {m_j+d-1\choose m_j}\frac{1}{n_j^{2m_j}}\\
\sum_{n_1,\ldots,n_k \in \mathbb{Z}_{\geq 1}\textup{ pairwise distinct}}\prod_{1 \leq j \leq k,\sum_j m_j=i} (3d/2)^{m_j}\frac{1}{n_j^{2m_j}}\\
=(3d/2)^i\sum_{n_1,\ldots,n_k \in \mathbb{Z}_{\geq 1}\textup{ pairwise distinct}}\prod_{1 \leq j \leq k,\sum_j m_j=i} \frac{1}{n_j^{2m_j}}\\
=(3d/2)^i(\sum_{n \geq 1}\frac{1}{n^2})^i=(\frac{3d}{2})^i(\frac{\pi^2}{6})^i.
\end{align*}
Therefore,
$$\rho_{d,2i}\leq (\frac{3d}{2})^i(\frac{\pi^2}{6})^i(\frac{1}{\pi^2})^i= (\frac{d}{4})^i.$$
\end{proof}
\begin{theorem}
We have the following estimates of $S_{p,d}$ and $T_{p,d}$:
$$S_{p,d}=\sum_{0 \leq i \leq d-2}\rho_{d,i}(\frac{\pi}{2})^{i-d}\lambda(d-i)p^{d-i}+R_{2,p,d}+R_{3,p,d}$$
and
$$T_{p,d}=\sum_{0 \leq i \leq d-2}\rho_{d,i}(\frac{\pi}{2})^{i-d}\beta(d-i)p^{d-i}+R'_{2,p,d}+R'_{3,p,d},$$
where $0 \leq R_{2,p,d}, R'_{2,p,d} \leq p^2, 0 \leq -R_{3,p,d},-R'_{3,p,d} \leq p(\frac{d}{4})^{d/2}$.
\end{theorem}
\begin{proof}
We set
$$z^d\rho_d(z)=\sum_{0 \leq i \leq d-2}\rho_{d,i}z^i+R_{1,d}(z)z^d.$$
Then for $z \in (0,\frac{\pi}{2})$,
$$0 \leq R_{1,d}(z)\leq \frac{\pi}{2z}.$$
Taking the sum of $\rho_d(z)$ over all $z=\frac{i\pi}{2p}$, $1 \leq i \leq p-1$, $i$ odd, we get
$$S_{p,d}=\sum_{0 \leq i \leq d-2}\rho_{d,i}(\frac{\pi}{2})^{i-d}\lambda_p(d-i)p^{d-i}+\sum_{1 \leq i \leq p-1, i\textup{ odd}}R_{1,d}(\frac{i\pi}{2p}).$$
We set
$$R_{2,p,d}=\sum_{1 \leq i \leq p-1, i\textup{ odd}}R_{1,d}(\frac{i\pi}{2p}),$$
$$R_{3,p,d}=\sum_{0 \leq i \leq d-2}\rho_{d,i}(\frac{\pi}{2})^{i-d}(\lambda_p(d-i)-\lambda(d-i))p^{d-i}.$$
Then
$$0 \leq R_{2,p,d}\leq\sum_{1 \leq i \leq p-1, i\textup{ odd}}\frac{p}{i} \leq p^2$$
and
\begin{align*}
0 \leq -R_{3,p,d}\leq \sum_{0 \leq i \leq d-2}\rho_{d,i}(\frac{\pi}{2})^{i-d}p^{1-d+i}p^{d-i}\\
\leq p\sum_{0 \leq i \leq d-2, i \textup{ even}}(\frac{d}{4})^{i/2}(\frac{\pi}{2})^{i-d}\leq \frac{p(\frac{d}{4})^{(d-2)/2}(\frac{\pi}{2})^{-2}}{1-\frac{4}{d}\frac{4}{\pi^2}}\leq p(\frac{d}{4})^{d/2}.
\end{align*}
Finally we have
$$S_{p,d}=\sum_{0 \leq i \leq d-2}\rho_{d,i}(\frac{\pi}{2})^{i-d}\lambda(d-i)p^{d-i}+R_{2,p,d}+R_{3,p,d},$$
so we get the result. The estimate for $T_{p,d}$ can be proved in the same way.
\end{proof}
\begin{remark}
For fixed $d$, this gives another expression of $S_{p,d}$ and $T_{p,d}$ up to $O(p)$: when $d$ is even,
$$S_{p,d}=\sum_{0 \leq i \leq d-2}\rho_{d,i}(\frac{\pi}{2})^{i-d}\lambda(d-i)p^{d-i}+O(p)$$
and when $d$ is odd,
$$T_{p,d}=\sum_{0 \leq i \leq d-2}\rho_{d,i}(\frac{\pi}{2})^{i-d}\beta(d-i)p^{d-i}+O(p)$$
This result is compatible with \Cref{4.7} by the relationship between the Bernoulli number, Euler number, Dirichlet lambda function and Dirichlet beta function.
\end{remark}

We sketch a proof of \Cref{5.1} in the large $p$ case. We have
\begin{align*}
b_{p,d}-1=\frac{Ap^d+Bp^{d-2}+Cp^{d-4}+O(p^{d-5})}{p^d-Dp^{d-2}-Ep^{d-4}+O(p^{d-5})}\\
=\frac{A+Bp^{-2}+Cp^{-4}+O(p^{-5})}{1-Dp^{-2}-Ep^{-4}+O(p^{-5})}\\
=A+\frac{(B+AD)p^{-2}+(C+AE)p^{-4}+O(p^{-5})}{1-Dp^{-2}-Ep^{-4}+O(p^{-5})}
\end{align*}
for some $A,B,C,D,E$, and we have $D>0,B+AD>0$. Note that even if the coefficients depend on $d$, the condition $p\geq 2^{d/4}$ and \Cref{5.9} will guarantee that the error term is no more than $p^{-5}$. We will prove separately that
$$p \mapsto (B+AD)p^{-2}+(C+AE)p^{-4}+O(p^{-5})$$
is decreasing and
$$p \mapsto 1-Dp^{-2}-Ep^{-4}+O(p^{-5})$$
is increasing on the set of odd numbers.

Now we give a detailed proof under the above idea. Denote
$$c_{d,i}=\rho_{d,i}(\frac{\pi}{2})^{i-d}\lambda(d-i),i \leq d-2$$
such that
$$S_{p,d}=\sum_{0 \leq i \leq d-2}c_{d,i}p^{d-i}+R_{2,p,d}+R_{3,p,d}.$$
In particular, $c_{d,0}=(\frac{\pi}{2})^{-d}\lambda(d) \in (0,1)$ and $1-2c_{d,0} \in (-1,1)$ for $d \geq 16$, and $c_{d,i}=0$ for odd $i$. We can write
\begin{align*}
S_{p,d}-S_{p,d-2}-c_{d,0}(p^d-2S_{p,d-2}-1)\\
=\sum_{0 \leq i \leq d-2}c_{d,i}p^{d-i}+R_{2,p,d}+R_{3,p,d}-(\sum_{0 \leq i \leq d-4}c_{d-2,i}p^{d-2-i}+R_{2,p,d-2}+R_{3,p,d-2})\\
-c_{d,0}(p^d-2\sum_{0 \leq i \leq d-4}c_{d-2,i}p^{d-2-i}-2R_{2,p,d-2}-2R_{3,p,d-2}-1)\\
=\sum_{2 \leq i \leq d-2}(c_{d,i}-c_{d-2,i-2}+2c_{d,0}c_{d-2,i-2})p^{d-i}+c_{d,0}+R_{4,p,d}\\
=(\frac{\pi}{2})^{2-d}\lambda(d-2)(\frac{d}{6}-1+2(\frac{\pi}{2})^{-d}\lambda(d))p^{d-2}\\
+(\frac{\pi}{2})^{4-d}\lambda(d-4)(\frac{5d^2+2d}{360}-\frac{d}{6}
+\frac{d}{6}2(\frac{\pi}{2})^{-d}\lambda(d))p^{d-4}\\
+\sum_{6 \leq i \leq d-2}(c_{d,i}-c_{d-2,i-2}+2c_{d,0}c_{d-2,i-2})p^{d-i}+c_{d,0}+R_{4,p,d},
\end{align*}
where
$$R_{4,p,d}=R_{2,p,d}+R_{3,p,d}-R_{2,p,d-2}-R_{3,p,d-2}+2c_{d,0}R_{2,p,d-2}+2c_{d,0}R_{3,p,d-2}.$$
We will write $(\frac{\pi}{2})^d(\frac{1}{2}(b_{p,d}-1)-c_{d,0})=G(p,d)/H(p,d)$, and prove the monotonicity of $G$ and $H$ below.

\textbf{Estimate of the numerator.} We set
\begin{align*}
G(p,d)=(\frac{\pi}{2})^d\frac{1}{p^d}(S_{p,d}-S_{p,d-2}-c_{d,0}(p^d-2S_{p,d-2}-1))\\
=(\frac{\pi}{2})^2\lambda(d-2)(\frac{d}{6}-1+2(\frac{\pi}{2})^{-d}\lambda(d))p^{-2}\\
+(\frac{\pi}{2})^4\lambda(d-4)(\frac{5d^2+2d}{360}-\frac{d}{6}
+\frac{d}{6}2(\frac{\pi}{2})^{-d}\lambda(d))p^{-4}\\
+\sum_{6 \leq i \leq d-2}(\frac{\pi}{2})^d(c_{d,i}-c_{d-2,i-2}+2c_{d,0}c_{d-2,i-2})p^{-i}+(\frac{\pi}{2})^dp^{-d}c_{d,0}+(\frac{\pi}{2})^dp^{-d}R_{4,p,d}.  
\end{align*}
Now we claim $G(p,d)>G(p+2,d)$, or equivalently, $G(p,d)-G(p+2,d)>0$. We see when $d \geq 16$, the following coefficient
$$(\frac{\pi}{2})^4\lambda(d-4)(\frac{5d^2+2d}{360}-\frac{d}{6}
+\frac{d}{6}2(\frac{\pi}{2})^{-d}\lambda(d))$$
and the coefficient $(\frac{\pi}{2})^dc_{d,0}$ are positive. Moreover,
$$(c_{d,i}-c_{d-2,i-2}+2c_{d,0}c_{d-2,i-2})\geq -c_{d-2,i-2}.$$
Since $0<p^{-i}-(p+2)^{-i}\leq 2ip^{-i-1}$, we see
$$(c_{d,i}-c_{d-2,i-2}+2c_{d,0}c_{d-2,i-2})(p^{-i}-(p+2)^{-i})\geq -c_{d-2,i-2}2ip^{-i-1}.$$

Therefore, to prove $G(p,d)-G(p+2,d)>0$, it suffices to verify
\begin{align*}
(\frac{\pi}{2})^2\lambda(d-2)(\frac{d}{6}-1+2(\frac{\pi}{2})^{-d}\lambda(d))(p^{-2}-(p+2)^{-2})\\
\geq \sum_{6 \leq i \leq d-2}(\frac{\pi}{2})^dc_{d-2,i-2}2ip^{-i-1}+(\frac{\pi}{2})^dp^{-d}R_{4,p,d}+(\frac{\pi}{2})^d(p+2)^{-d}R_{4,p+2,d}. \tag{A7}
\end{align*}
Let $\alpha \in (0,1)$, we have
$$6\alpha^6+7\alpha^7+\ldots=\frac{6\alpha^6-5\alpha^7}{(1-\alpha)^2}\leq \frac{6\alpha^6}{(1-\alpha)^2}.$$
We have the following estimate:
$$(\frac{\pi}{2})^2\lambda(d-2)(\frac{d}{6}-1+2(\frac{\pi}{2})^{-d}\lambda(d))\geq (\frac{\pi}{2})^2\cdot \frac{5}{3}=\frac{5\pi^2}{12},$$
$$p^{-2}-(p+2)^{-2}=4(p+1)p^{-2}(p+2)^{-2}\geq 2p^{-3}.$$
This means the left side of (A7) is at least $\frac{5\pi^2}{6}p^{-3}$. We have
$$c_{d,i}=\rho_{d,i}(\frac{\pi}{2})^{i-d}\lambda(d-i)\leq (\frac{d}{4})^{\frac{i}{2}}(\frac{\pi}{2})^{i-d}\cdot 2,$$
\begin{align*}
\sum_{6 \leq i \leq d-2}(\frac{\pi}{2})^dc_{d-2,i-2}2ip^{-i-1}\leq  \sum_{6 \leq i \leq d-2}4(\frac{d}{4})^{\frac{i-2}{2}}(\frac{\pi}{2})^iip^{-i-1}\leq 4(\frac{d}{4})^2(\frac{\pi}{2})^6p^{-7}\frac{6}{(1-\frac{1}{p}\frac{\sqrt{d}}{2}\frac{\pi}{2})^2},
\end{align*}
\begin{align*}
|R_{4,p,d}|\leq |R_{2,p,d}|+|R_{3,p,d}|+|(1-2c_{d,0})R_{2,p,d-2}|+|(1-2c_{d,0})R_{3,p,d-2}|\\
\leq |R_{2,p,d}|+|R_{3,p,d}|+|R_{2,p,d-2}|+|R_{3,p,d-2}|
\leq 2p^2+2p(\frac{d}{4})^{\frac{d}{2}},
\end{align*}
\begin{align*}
(\frac{\pi}{2})^dp^{-d}|R_{4,p,d}|\leq 2p^{2-d}(\frac{\pi}{2})^d+2p^{1-d}(\frac{\pi}{2})^d(\frac{d}{4})^{\frac{d}{2}}.
\end{align*}
Since $(p+2)^{1-d}<p^{1-d}$, we also have
\begin{align*}
(\frac{\pi}{2})^d(p+2)^{-d}|R_{4,p+2,d}|\leq 2p^{2-d}(\frac{\pi}{2})^d+2p^{1-d}(\frac{\pi}{2})^d(\frac{d}{4})^{\frac{d}{2}},
\end{align*}
\begin{align*}
\frac{1}{p}\frac{\sqrt{d}}{2}\frac{\pi}{2}\leq \frac{1}{4},\frac{1}{(1-\frac{1}{p}\frac{\sqrt{d}}{2}\frac{\pi}{2})^2}\leq \frac{16}{9}<2. 
\end{align*}
The above inequalities imply the right side of (A7) is smaller than
$$8(\frac{d}{4})^2(\frac{\pi}{2})^6p^{-7}\cdot 6+4p^{2-d}(\frac{\pi}{2})^{d}+4p^{1-d}(\frac{\pi}{2})^d(\frac{d}{4})^{\frac{d}{2}}.$$
So to prove (A7), it suffices to prove
\begin{align*}
\frac{5\pi^2}{6}p^{-3}\geq 8(\frac{d}{4})^2(\frac{\pi}{2})^6p^{-7}\cdot 6+4p^{2-d}(\frac{\pi}{2})^{d}+4p^{1-d}(\frac{\pi}{2})^d(\frac{d}{4})^{\frac{d}{2}}.
\end{align*}
When $p \geq 2^{d/4}$, we have
$$\frac{d}{4}\leq \log_2p,(\frac{\pi}{2})^d\leq p^{4\log_2(\pi/2)},(\frac{d}{4})^{\frac{d}{2}}\leq p^{2\log_2(d/4)}.$$
Now it suffices to prove for $p \geq 2^{d/4}$, $d \geq 16$,

$$\frac{5\pi^2}{6}p^{-3}\geq 8(\log_2p)^2(\frac{\pi}{2})^6p^{-7}\cdot 6+4p^{2-d+4\log_2(\pi/2)}+4p^{1-d+4\log_2(\pi/2)+2\log_2(d/4)}.$$
Equivalently,
$$\frac{5\pi^2}{6}\geq 8(\log_2p)^2(\frac{\pi}{2})^6p^{-4}\cdot 6+4p^{5-d+4\log_2(\pi/2)}+4p^{4-d+4\log_2(\pi/2)+2\log_2(d/4)}.$$
The right side of this inequality is decreasing with respect to both $p,d \geq 16$, so it suffices to verify this inequality when $p,d=16$, which is the case since the right side is approximately 0.176. Therefore, we have proved that $p \mapsto G(p,d)$ is decreasing.

\textbf{Estimate of the denominator.} Set
\begin{align*}
H(p,d)=\frac{1}{p^d}(p^d-2S_{p,d-2}-1)\\
=1-2\sum_{0 \leq i \leq d-4}c_{d-2,i}p^{-2-i}-\frac{2}{p^d}R_{2,p,d-2}-\frac{2}{p^d}R_{3,p,d-2}-\frac{(-1)^{rd}}{p^d}
\end{align*}
and
$$H_1(p,d)=c_{d-2,0}p^{-2}+\frac{1}{p^d}R_{2,p,d-2}+\frac{1}{p^d}R_{3,p,d-2}+\frac{(-1)^{rd}}{2p^d}.$$
Here the term $(-1)^{rd}$ is always $1$ when $d$ is even. However, the same proof works for odd $d$ where we replace $S_{p,d}$ with $T_{p,d}$, so we will include this term here.

We claim $p \mapsto H(p,d)$ is increasing. We see all coefficients $c_{d-2,i}$ are positive and $c_{d-2,0}=(\frac{\pi}{2})^{2-d}\lambda(d-2)$, so to prove $H(p,d)<H(p+2,d)$, it suffices to prove $H_1(p,d)>H_1(p+2,d)$. Then, it suffices to prove
\begin{align*}
(\frac{\pi}{2})^{2-d}\lambda(d-2)(p^{-2}-(p+2)^{-2})\\
\geq \frac{1}{p^d}R_{2,p,d-2}+\frac{1}{p^d}R_{3,p,d-2}+\frac{1}{2p^d}-\frac{1}{(p+2)^d}R_{2,p+2,d-2}-\frac{1}{(p+2)^d}R_{3,p+2,d-2}+\frac{1}{2(p+2)^d}.
\end{align*}
We can use the same estimate above for $p^{-2}-(p+2)^{-2}$, $R_{2,p,d}\geq 0$, $R_{3,p,d} \leq 0$ when we prove $G(p,d)>G(p+2,d)$. So it suffices to verify
\begin{align*}
(\frac{\pi}{2})^{2-d}2p^{-3}
\geq p^{2-d}+p^{1-d}(\frac{d}{4})^{\frac{d}{2}}+\frac{1}{p^d}.
\end{align*}
Also,
$$(\frac{\pi}{2})^d\leq p^{4\log_2(\pi/2)},(\frac{d}{4})^{\frac{d}{2}}\leq p^{2\log_2(d/4)}.$$
So it suffices to verify
\begin{align*}
(\frac{\pi}{2})^{2}2p^{-3-4\log_2(\pi/2)}
\geq p^{2-d}+p^{1-d+2\log_2(d/4)}+p^{-d},
\end{align*}
which is equivalent to
\begin{align*}
2(\frac{\pi}{2})^{2}
\geq p^{5-d+4\log_2(\pi/2)}+p^{4-d+4\log_2(\pi/2)+2\log_2(d/4)}+p^{3+4\log_2(\pi/2)-d}.
\end{align*}
The right side of this inequality is decreasing with respect to both $p,d \geq 16$, so it suffices to verify this inequality when $p,d=16$, which is the case since the right side is approximately $3.20\cdot 10^{-7}$. Therefore, we have proved that $p \mapsto H(p,d)$ is increasing.

In sum, we have proved that in the large $p$ case, $p\mapsto G(p,d)$ is decreasing and $p \mapsto H(p,d)$ is increasing on the set of odd integers $p \geq 16$. Thus
$$p \mapsto b_{p,d}-1=2c_{d,0}+\frac{2(\frac{\pi}{2})^{-d}G(p,d)}{H(p,d)}$$
is decreasing on the same set. 

The case $d$ odd is proved similarly where we replace $\lambda$-function with $\beta$-function and use $1/2\leq \beta(d) \leq 1$ instead of $1 \leq \lambda(d)\leq 2$. We just need to prove two stronger inequalities
$$\frac{5\pi^2}{12}\geq 8(\log_2p)^2(\frac{\pi}{2})^6p^{-4}\cdot 6+4p^{5-d+4\log_2(\pi/2)}+4p^{4-d+4\log_2(\pi/2)+2\log_2(d/4)}$$
and
\begin{align*}
(\frac{\pi}{2})^{2}
\geq p^{5-d+4\log_2(\pi/2)}+p^{4-d+4\log_2(\pi/2)+2\log_2(d/4)}+p^{3+4\log_2(\pi/2)-d},
\end{align*}
which still hold in this case. So the large $p$ case for even $d$ is proved, which concludes the proof of \Cref{5.1}.

\section*{Acknowledgement}

The author would like to thank Ilya Smirnov for helpful suggestions and showing the computational results in subsection 5.1. The author would like to thank Joel Castillo-Rey and Nick Cox-Steib for helpful suggestions.

\bibliographystyle{plain}
\bibliography{reference}

\end{document}